\documentclass[]{amsart}

\usepackage[]{graphicx,color}

\makeatletter
\@addtoreset{equation}{section}

\makeatother

\newtheorem{theorem}{Theorem}[section]
\newtheorem{proposition}[theorem]{Proposition}

\newtheorem{lemma}[theorem]{Lemma}
\theoremstyle{definition}
\newtheorem{definition}[theorem]{Definition}

\theoremstyle{remark}
\newtheorem{remark}[theorem]{Remark}

\newcommand{\llbracket}{[\![}
\newcommand{\rrbracket}{]\!]}
\DeclareMathOperator{\sech}{sech}
\DeclareMathOperator{\cn}{cn}
\DeclareMathOperator{\dn}{dn}
\DeclareMathOperator{\sn}{sn}
\DeclareMathOperator{\sgn}{sign}

\begin{document}

\title{Elastic curves and phase transitions}
\author{Tatsuya Miura}
\address{Graduate School of Mathematical Sciences, The University of Tokyo, 3-8-1 Komaba, Meguro, Tokyo, 153-8914 Japan}
\email{miura@ms.u-tokyo.ac.jp}
\address{Max Planck Institute for Mathematics in the Sciences, Inselstra{\ss}e 22, 04103 Leipzig, Germany}
\email{tmiura@mis.mpg.de}

\begin{abstract}
	This paper is devoted to classical variational problems for planar elastic curves of clamped endpoints, so-called Euler's elastica problem.
	We investigate a straightening limit that means enlarging the distance of the endpoints, and obtain several new results concerning properties of least energy solutions.
	In particular we reach a first uniqueness result that assumes no symmetry.
	As a key ingredient we develop a foundational singular perturbation theory for the modified total squared curvature energy.
	It turns out that our energy has almost the same variational structure as a phase transition energy of Modica-Mortola type at the level of a first order singular limit.
\end{abstract}

\maketitle


\section{Introduction}\label{sectintroduction}

In this paper we consider the most classical variational model of inextensible flexible rods studied from the 18th century, so-called {\em Euler's elastica problem}: Among smooth planar curves $\gamma$ of fixed length subject to the clamped boundary condition, we minimize the total squared curvature, also known as {\it bending energy},
\begin{eqnarray}\label{eqntotalsquaredcurvature}
\mathcal{B}[\gamma]=\int_{\gamma}\kappa^2ds,
\end{eqnarray}
where $s$ denotes the arc length parameter and $\kappa$ denotes the curvature.
The clamped boundary condition means that the endpoints are prescribed up to the first order derivatives (Figure \ref{figcurvebc}).

\begin{figure}[tb]
	\centering
	\includegraphics[width=60mm]{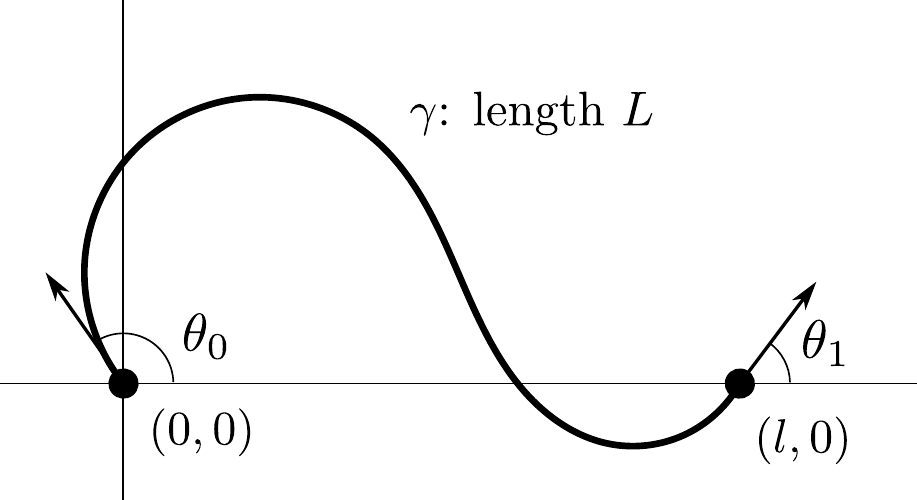}
	\caption{Clamped curve.}
	\label{figcurvebc}
\end{figure}

Euler's elastica problem is of one-dimension, and in particular the existence and regularity of solutions are by now standard.
However this problem is of higher-order, nonconvex, strongly nonlinear, and possessing a nonlocal constraint, so that not only quantitative but also qualitative properties of solutions sensitively depend on boundary data.
For these reasons, even today no rigorous method is available for obtaining uniqueness or the exact shapes of solutions just from given boundary data, except for special cases, e.g., closed curves.

The aim of this paper is to derive uniqueness as well as some properties of solutions for certain boundary data.
Throughout this paper we focus on global minimizers (least energy solutions) in a {\em straightening} regime, which roughly means that the distance of the endpoints is large.
We summarize our main results without precise definitions; one may refer Figure~\ref{figcurvebc} at this stage (see Section 2 for details).

\begin{theorem}\label{thmmain0}
	Fix any length $L>0$ and boundary angles $\theta_0,\theta_1\in(-\pi,\pi)$ with $\theta_0\theta_1<0$.
	Then there exists $\bar{l}\in(0,L)$ such that if the distance of the endpoints $l$ satisfies $l\in(\bar{l},L)$, Euler's elastica problem admits a unique global minimizer $\gamma_l$, and its tangential angle is strictly monotone from $\theta_0$ to $\theta_1$ (in particular, $\gamma_l$ is convex).
	Moreover, as $l\to L$, the rescaled curve $\hat{\gamma}_l$ defined as
	$$\hat{\gamma}_l(\hat{s}):=\frac{1}{\varepsilon_l}\gamma_l(\varepsilon_l\hat{s}), \quad \varepsilon_l:=\frac{L-l}{4\sqrt{2}\left(\sin^2\frac{\theta_0}{4}+\sin^2\frac{\theta_1}{4}\right)},$$
	converges locally smoothly to the borderline elastica with initial angle $\theta_0$.
	By symmetry, a similar convergence holds at the other endpoint.
\end{theorem}

The {\em borderline elastica} is defined in Definition \ref{defborderline0} (see also Figure \ref{figborderline}).
The condition $\theta_0\theta_1<0$ is essential for the convexity.
All the results, even uniqueness, are totally new: A particular feature is that we deal with non-closed curves of no symmetry, and impose no assumptions except for boundary data.
It should be noted that the uniqueness of global minimizers is not valid for general boundary data, and thus naturally requires some suitable assumptions.

Toward Theorem \ref{thmmain0} we mainly consider an auxiliary problem that is minimization of the modified total squared curvature,
\begin{eqnarray}\label{eqnmodifiedtotalsquaredcurvature}
\mathcal{E}_\varepsilon[\gamma]=\varepsilon^2\int_{\gamma}\kappa^2ds+\int_{\gamma}ds,
\end{eqnarray}
under the same clamped boundary condition but no length constraint.
For clarification we call the first one {\it inextensible problem} and the second {\it extensible problem}.
For the extensible problem we first carry out an asymptotic analysis as $\varepsilon\to0$, and apply it in order to obtain more rigid properties, as convexity, of minimizers for small $\varepsilon$: Up to this step we simultaneously deal with convex and nonconvex cases.
Then, independently developing a simple method that ensures some uniqueness among convex curves, we prove the uniqueness of global minimizers in a convexly straightening regime, i.e., $\theta_0\theta_1<0$ and $\varepsilon$ small.
We finally reach Theorem \ref{thmmain0} by reducing the inextensible problem in the limit $l\to L$ to the extensible problem in $\varepsilon\to0$, with the help of our uniqueness result.

A key step of our argument is the asymptotic analysis part, in which we discover a theoretical connection between two classical theories; elastic curves and phase transitions (see Section \ref{subsectphasetransition}).
Although the present paper focuses on the most classical planar framework, this insight is of independent interest and expected to be also applicable to e.g.\ free boundary problems and higher-codimensional problems, which will be addressed in our future works.

This paper is organized as follows.
In the rest of this section we present more precise backgrounds and our motivation, results, and methods.
All the main results of this paper are collected in Section \ref{sectmainresults}.
The results are sequentially proved in Sections \ref{sectenergyexpansion}, \ref{sectconvergence}, \ref{sectqualitativeproperty}, and \ref{sectrelation}.

\subsection{Euler's elastica: the origin}

Euler's elastica problem (inextensible problem) is first studied by Euler in 1744 \cite{Euler1744} and since then extensively studied in various fields.
In his celebrated study, Euler derives ODEs for solution curves (i.e., critical points) and moreover classifies the types of solution curves qualitatively.
The solution curves are nowadays called {\it Euler's elasticae} or simply {\em elasticae}.
See e.g.\ \cite{Fr91,Le08,Lo44,Ma10,Sa08,SaLe10,Tr83} for more details of the history.

We recall some basic facts on solution curves, using modern terms.
By the classical Lagrange multiplier method, for any critical point $\gamma$ in the inextensible problem, there is a multiplier $\lambda\in\mathbb{R}$ such that $\gamma$ is also a critical point of the energy
$$\int_{\gamma}\kappa^2ds+\lambda\int_\gamma ds$$
under the same boundary condition.
Calculating the first variation, we find that the signed curvature $\kappa$ of $\gamma$ satisfies
\begin{eqnarray}\label{eqnelasticalambda}
2\partial_s^2\kappa+\kappa^3-\lambda\kappa=0,
\end{eqnarray}
which we call {\it elastica equation} in this paper.
The elastica equation is uniquely solved for any given $\lambda$ and initial values $\kappa(0)$, $\kappa'(0)$ \cite{Li96}.
All solutions are expressed in terms of the Jacobi elliptic functions: Figure~\ref{figelasticae_sol} exhibits a classification of basic patterns, and the cases (i)--(iii) and (iii)--(v) respectively correspond to the $\dn$- and $\cn$-function.
In particular, the case (iii) is called {\it critical} or {\it borderline} elastica, and the only case having no periodicity: The convergence limit in Theorem \ref{thmmain0} is nothing but a part of this curve.
See e.g.\ \cite{AuPo10,Br92,Lo44,MlHa17,Si08} for more details.

\begin{figure}[tb]
	\centering
	\includegraphics[width=110mm]{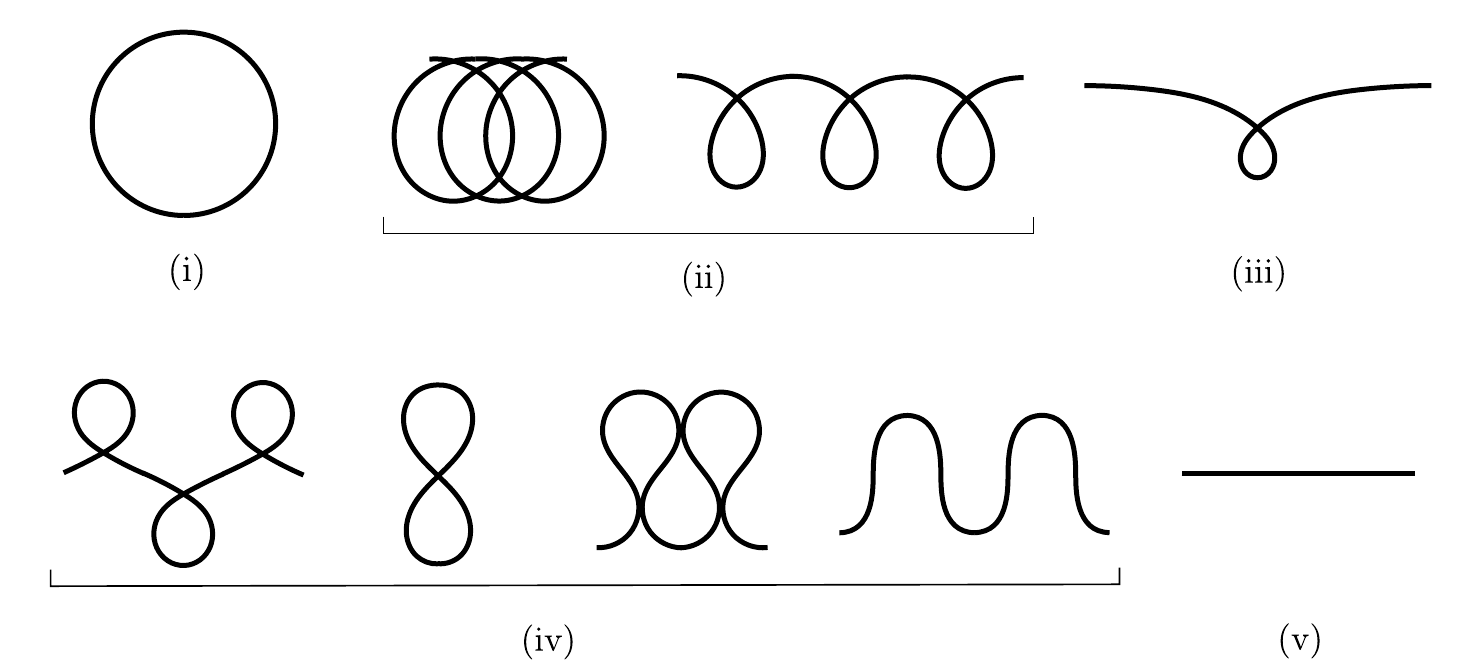}
	\caption{Basic patterns of elasticae.}
	\label{figelasticae_sol}
\end{figure}

\subsection{Shape of clamped elastica: remaining problems}\label{subsectshapeofelastica}

Notwithstanding that the elastica problem is already solved at the level of ``equation'' as above, it is still difficult to study our ``boundary value problem''.

One reason is that our clamped boundary condition does not fix any of the parameters $\lambda$, $\kappa(0)$, and $\kappa'(0)$; in particular there are infinitely many local minimizers (stable critical points) with different shapes as e.g.\ in Figure~\ref{figloops} (see Appendix \ref{appendixexistence}).
Although there are useful comprehensive formulae describing the relations between our boundary condition and solution curves \cite{Li98,Li98b}, the formulae are given as involved simultaneous transcendental equations (including elliptic functions and elliptic integrals) and not necessarily direct evidence for a clear understanding of the shapes of solution curves.

In addition, even if we know all candidates of minimizers (as local minimizers) in some sense, there remain intricate calculations of their energies for seeking global minimizers.
As mentioned earlier, the uniqueness of global minimizers is not expected in general due to several reasons, which are still unclear (cf.\ \cite{SaSa14}).
To the best of the author's knowledge, there is no comprehensive uniqueness result in the literature, although in some cases of special symmetry the uniqueness is easily proved or disproved as explained in Appendix \ref{appendixuniqueness}.

The only case that all critical points and their local and global optimality are fully understood would be closed curves ($l=0$ and $\theta_0=\theta_1$):
In this case it is shown in \cite{AvKaSo13,LaSi85,Sa12} that any critical point is an $n$-fold circle or an $n$-fold figure-of-eight, any local minimizer is an $n$-fold circle or the $1$-fold figure-of-eight, and a global minimizer is the $1$-fold circle.
Non-closed curves are substantially less tractable since scaling arguments do not reduce multipliers.
Incidentally, we mention that the case of buckling as in Figure \ref{figbuckling} is also well studied since this case is covered by small deformation theory so that the bifurcation theory can be applied (see e.g.\ \cite{An83,An95,Ma84}).
In this paper we are concerned with more nonlinear phenomena.


\begin{figure}[tb]
	\centering
	\includegraphics[width=80mm]{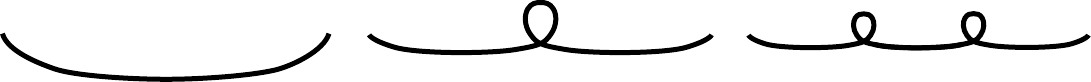}
	\caption{Loops.}
	\label{figloops}
\end{figure}

Upon exploring the shapes of elastic curves, one of the most interesting questions is to ask the number of infection points (i.e., sign-changing points of the curvature) since it is concerned with not only the shape but also other properties or mathematical techniques directly.
As a pioneering work with respect to inflection points, in 1906, Born proved that any solution curve without inflection point is stable \cite{Born1906}.
Recently, Sachkov et al.\ \cite{Sa08,Sa08b,SaSa14} revisit the elastica problem in view of optimal control and obtains several new results about local and global optimality.
In particular, it is shown that any stable solution has at most two inflection points.
The upper bound two is optimal even for global minimizers, since a well-known buckling example as in Figure~\ref{figbuckling} may be a global minimizer in a certain case.
However, there are few results to exactly determine the number of inflection points from given boundary data.

\begin{figure}[tb]
	\centering
	\includegraphics[width=50mm]{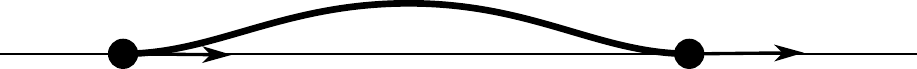}
	\caption{Buckling.}
	\label{figbuckling}
\end{figure}

In the rest of this subsection, to clarify the above difficulties, we observe an example case of straightening by seeing Figure~\ref{figelasticaegray}, which is just formal and heuristic.
Each row in Figure~\ref{figelasticaegray} describes ``continuous'' deformation from a closed elastica, where the continuity especially means that the ``winding number'' is preserved.
Since the two boundary angles are given to be same in this figure, the dotted curves have the same energies as the corresponding bold curves, respectively.
The gray region in Figure~\ref{figelasticaegray} indicates expected global minimizers.
The point (a) indicates a change of winding number: Such a point would exist since a closed global minimizer is known to be a circle, while a straightened global minimizer would be an ``S-shaped'' curve (in view of our result below).
The point (b) indicates a symmetry breaking, and accordingly a change of the number of global minimizers.
The number of inflection points would also change at (a) and (b).
These nontrivial transitions are expected at least formally and would be obstruction in our analysis.
To the author's knowledge, there is no general result to determine typical points as (a) or (b) rigorously (cf.\ \cite{ArSa09,GaLeMe80,SaSa14}).

\begin{figure}[tb]
	\centering
	\includegraphics[width=110mm]{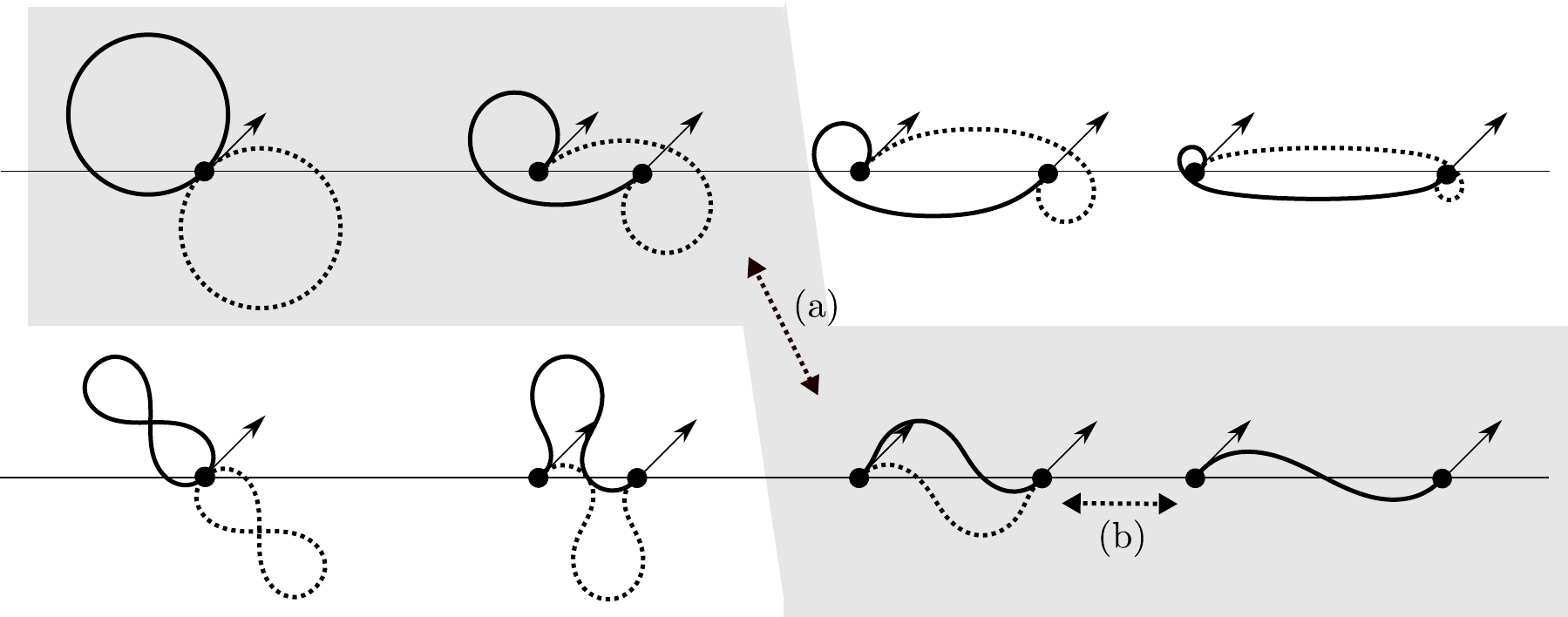}
	\caption{Formal observation for a straightening process.}
	\label{figelasticaegray}
\end{figure}

\subsection{Straightened elastica: main results}

The purpose of this paper is to obtain the above-mentioned properties of global minimizers only under assumptions on the boundary condition.
As observed above, finding an appropriate assumption is as important as obtaining assertions.
This paper focuses on a straightening regime, i.e., the case that the distance of the endpoints is close to the length of a curve (excluding buckling).
In this case we are able to detect global minimizers since the main (singular) effect of energy is localized near the endpoints and hence we can significantly reduce the involved competition of energy between candidates of a global minimizer (critical points).

However, even if we focus on the straightening regime, the inextensible problem is not easy to tackle directly.
The main reason is the length constraint, which is nonlocal and causes a multiplier.
To circumvent this difficulty, in this paper, we first consider the singular limit $\varepsilon\to0$ for the extensible problem.
Considering this limit is physically natural.
In fact, the constant $\varepsilon^2$ is interpreted as bending rigidity divided by tension, and we expect that straightened elastic curves have very high tension.
The extensible problem is relatively tractable in the sense that the multiplier $\varepsilon^2$ is a priori fixed.

Our main result (Theorem \ref{thmmain1}) states that, in the extensible problem, any sequence of global minimizers is straightened as $\varepsilon\to0$ as in Figure~\ref{figcurvethm} for an arbitrary given boundary condition.
More precisely, for small $\varepsilon$, any minimizer bends at the scale of $\varepsilon$ near the endpoints and is almost straight elsewhere, i.e., the tangent vectors are almost rightward.
In addition, a rescaled convergence limit at an endpoints is a part of the borderline elastica.
In this proof we use a theoretical analogy to the phase transition theory, as explained in the next subsection.
Our result also implies other more qualitative properties (Theorem \ref{thmqualitativeelastica}).
For instance, as a direct corollary, we find that any minimizer has no self-intersection for any small $\varepsilon$.
In addition, combining our result with expressions of the curvatures by elliptic functions, we determine the number of inflection points for generic boundary angles (except some critical cases).
The number is zero or one, depending only on the signs of boundary angles as in Figure~\ref{figcurvecor}.
Furthermore, in the case of no inflection point, we prove the uniqueness of global minimizers (Theorem \ref{thmuniqueness}).
Our proof uses a change of variables based on the Gauss map, which has been already used by Born \cite{Born1906} for studying stability of convex curves.
This change of variables yields a ``convexification'' under some a priori convexity, which is ensured by Theorem \ref{thmuniqueness}, so that our problem is translated into minimization of a convex energy defined on a convex set.
We thus discover that Born's convexification is also a powerful tool for uniqueness, which is a global problem in contrast to stability.

\begin{figure}[tb]
	\centering
	\includegraphics[width=60mm]{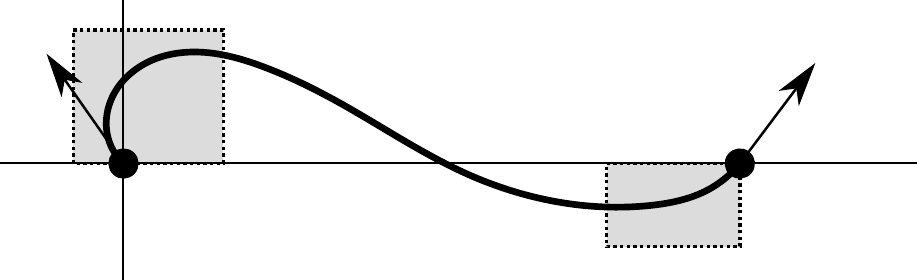}
	\caption{Straightened elastica.}
	\label{figcurvethm}
\end{figure}

\begin{figure}[tb]
	\centering
	\includegraphics[width=100mm]{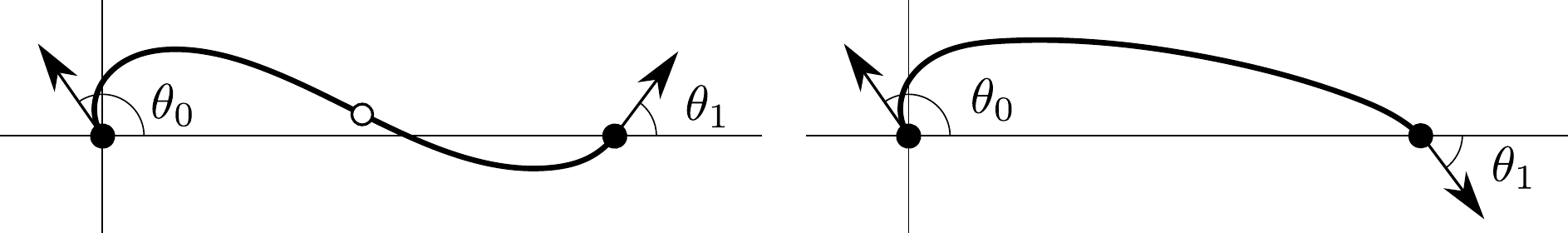}
	\caption{Straightened elasticae with and without inflection point.}
	\label{figcurvecor}
\end{figure}

We finally translate the above results on the extensible problem as $\varepsilon\to0$ back to the inextensible problem as $l\to L$.
Generally speaking, the Lagrange multiplier method indicates some kind of relation between the extensible and inextensible problems at the level of critical points.
In this paper, we investigate the precise relation at the level of global minimizers.
At this time the translation is partial (Theorem \ref{thmmain2}) due to the lack of uniqueness theory, but fully proved in a convex case (Theorem \ref{thmmain2'}) thanks to our uniqueness result.
This is the reason why Theorem \ref{thmmain1} covers only convex curves.

\subsection{Phase transition: a new perspective}\label{subsectphasetransition}

As mentioned, our asymptotic analysis is based on a new perspective that indicates a theoretical connection between the (extensible) elastic problem and the phase transition theory.

We briefly recall the studies on phase transition energy.
The minimizing problem of a potential energy perturbed by a gradient term, as
$$E_\epsilon[u]=\epsilon^2\int_\Omega|\nabla u|^2+ \int_\Omega W(u),$$
has been widely studied, in particular, in view of the van der Waals-Cahn-Hilliard theory of phase transitions \cite{CaHi58,VdW79}.
Here $\Omega\subset\mathbb{R}^n$ is a certain open set.
The potential function $W$ is often taken as the double-well potential $W(u)=(1-u^2)^2$, and the volume constraint $\int_\Omega u=M$ is often imposed.

In the phase transition problem, for small $\epsilon$, the values of a minimizer should be almost separated into the phases $1$ and $-1$ to minimize the potential energy.
Moreover, if a minimizer needs to have a transition between the two phases due to the volume constraint, then the area of ``interface'' is expected to be minimized due to the effect of perturbation.
These expectations are proved by Carr-Gurtin-Slemrod \cite{CaGuSl84} in a one-dimensional case, and by Modica \cite{Mo87} and Sternberg \cite{St88} in higher-dimensional cases.
The higher-dimensional results \cite{Mo87,St88} are described in terms of $\Gamma$-convergence, which is introduced by De Giorgi in 1970's (see e.g.\ \cite{Br14,DM93}).
The $\Gamma$-convergence result particularly implies the first order expansion of the minimum value of $E_\epsilon$ as $\epsilon\to0$.
Moreover, it also implies that, up to a subsequence, any sequence of minimizers $u_\epsilon$ converges in $L^1$ to a characteristic function $u_0\in BV(\Omega;\{-1,1\})$ of which total variation is minimized among functions $u\in BV(\Omega;\{-1,1\})$ with $\int_\Omega u=M$.
Some stronger convergence results are also known, even for local minimizers \cite{CaCo95} or critical points \cite{HuTo00} with certain boundedness; roughly speaking, a locally uniform convergence holds except interfaces.
Furthermore, at least formally, one expects that the transition part of a minimizer is close to a rescaled ``transition layer'' solution.
In fact, in the particular case that $\Omega=(-1,1)$ and $M=0$, it is easy to prove that the rescaled minimizer $\hat{u}_\epsilon(x)=u_\epsilon(\epsilon x)$ is nothing but a transition layer, i.e., a solution to $|u'|^2=W(u)$, as in Figure~\ref{figtransitionlayer}.
Finally, it should be mentioned that a basic strategy for the above higher-dimensional results \cite{Mo87,St88} has been provided in the earlier paper by Modica and Mortola \cite{MoMo77}.
The paper deals with an unconstrained problem for the periodic potential $W(u)=\sin^2(\pi u)$; this potential is more directly relative to our problem.

\begin{figure}[tb]
	\centering
	\includegraphics[width=110mm]{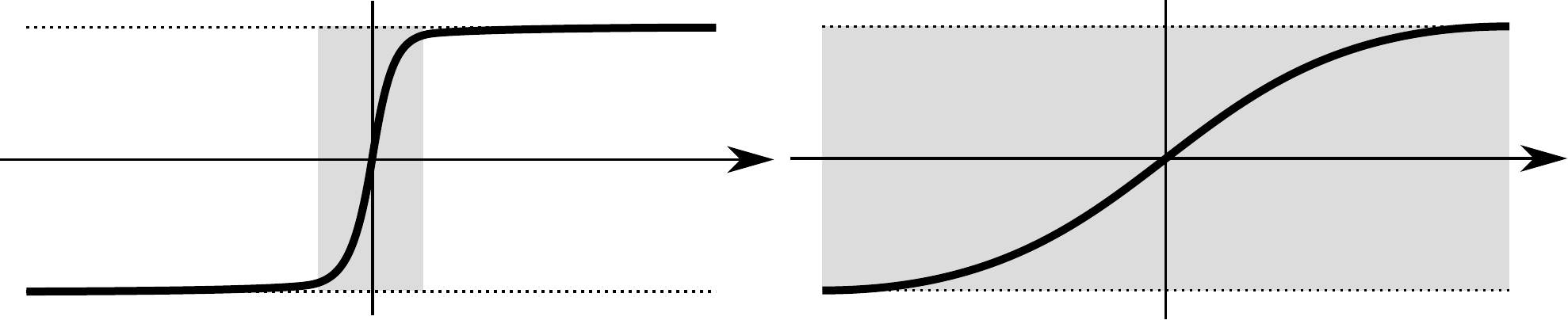}
	\caption{A minimizer $u_\varepsilon$ and a transition layer.}
	\label{figtransitionlayer}
\end{figure}

We shall go back to our elastic curve problem.
For a curve $\gamma$ as in Figure~\ref{figcurvebc}, we denote its length by $L$ and represent the modified total squared curvature in terms of its tangential angle function $\vartheta:[0,L]\to\mathbb{R}$ (i.e., $\partial_s\gamma=(\cos\vartheta,\sin\vartheta)$) as
\begin{eqnarray}
\mathcal{E}_\varepsilon[\gamma] &=& \varepsilon^2\int_0^L|\partial_s\vartheta|^2ds + \int_0^Lds \nonumber\\
&=& \varepsilon^2\int_0^L|\partial_s\vartheta|^2ds + \int_0^L(1-\cos\vartheta)ds + l, \nonumber
\end{eqnarray}
where $l$ is the fixed distance of the endpoints.
The last equality follows since $\int_0^L\cos\vartheta ds$ is nothing but the difference of the $x$-coordinates at the endpoints.
The above expression indicates that $\mathcal{E}_\varepsilon$ can be regarded as a one-dimensional phase transition energy with the periodic potential $W(\theta)=1-\cos\theta$ ($=2\sin^2(\theta/2)$).
All the stable phases $\theta\in2\pi\mathbb{Z}$ correspond to the rightward tangent vector.

By using this observation, we obtain the first order expansion of the energy minimum as $\varepsilon\to0$ (Lemma \ref{lemasymptitoticmodified}), which is essentially similar to the phase transition problem.
Of course, there are some differences between our elastic problem and the original phase transition problem; the integration interval $[0,L]$ is not fixed, and there are a greater number of constraints due to the clamped boundary condition, namely,
$$\vartheta(0)=\theta_0, \quad \vartheta(L)=\theta_1\ (\textrm{mod }2\pi), \quad \int_0^L\cos\vartheta(s)ds=l, \quad \int_0^L\sin\vartheta(s)ds=0,$$
than the above volume constraint.
However, our result reveals that the differences do not affect the expansion up to the first order (but would do from the next order).
The expansion is proved by standard steps in the calculus of variations, in particular, in the phase transition theory; we first obtain a lower bound of the energy, and then construct a suitable sequence of test curves that ensures the optimality of the lower bound.
In our construction of test curves, we use the fact that the lengths of curves are unconstrained.

Then we show that the first order expansion is enough sharp to reveal the precise convergence of minimizers as $\varepsilon\to0$.
In particular, near the endpoints, the rescaled tangential angles smoothly converge to a part of transition layer (Figure~\ref{figtransitionlayer3}), i.e., a solution to the transition layer equation
\begin{eqnarray}\label{eqntransitionlayer}
|\partial_s\vartheta|^2=1-\cos\vartheta.
\end{eqnarray}
The curve corresponding to the transition layer has one loop, and is nothing but the borderline elastica (Figure~\ref{figelasticae_sol} (iii)).
Thus we give a new interpretation of this typical elastica in view of the phase transition theory.
In our proof, the rescaled convergence is first justified in a weak sense, and then the regularity of convergence is improved by using the fact that any rescaled minimizer satisfies the rescaled elastica equation, that is, (\ref{eqnelasticalambda}) with $\lambda=1$.
Although the scaling law is enough just for ensuring convergence of rescaled solutions, the prefactor of the first order expansion plays a key role in the proof that the convergent limit must satisfy not only the rescaled elastica equation but also the transition layer equation (\ref{eqntransitionlayer}).
Our proof is based on the one-dimensionality in the sense that we use a one-dimensional partition of the domain of curves.
This type of convergence result seems to be new even for the phase transition model.

We note that our study is also essentially related to the concept of $\Gamma$-convergence although this paper includes no explicit statement.
One may obtain a more general $\Gamma$-convergence result such that the function space of the limits of the tangential angles contains general $2\pi\mathbb{Z}$-valued $BV$-functions, but we do not state it in this paper to avoid digressing from our main subject.

\begin{figure}[tb]
	\centering
	\includegraphics[width=100mm]{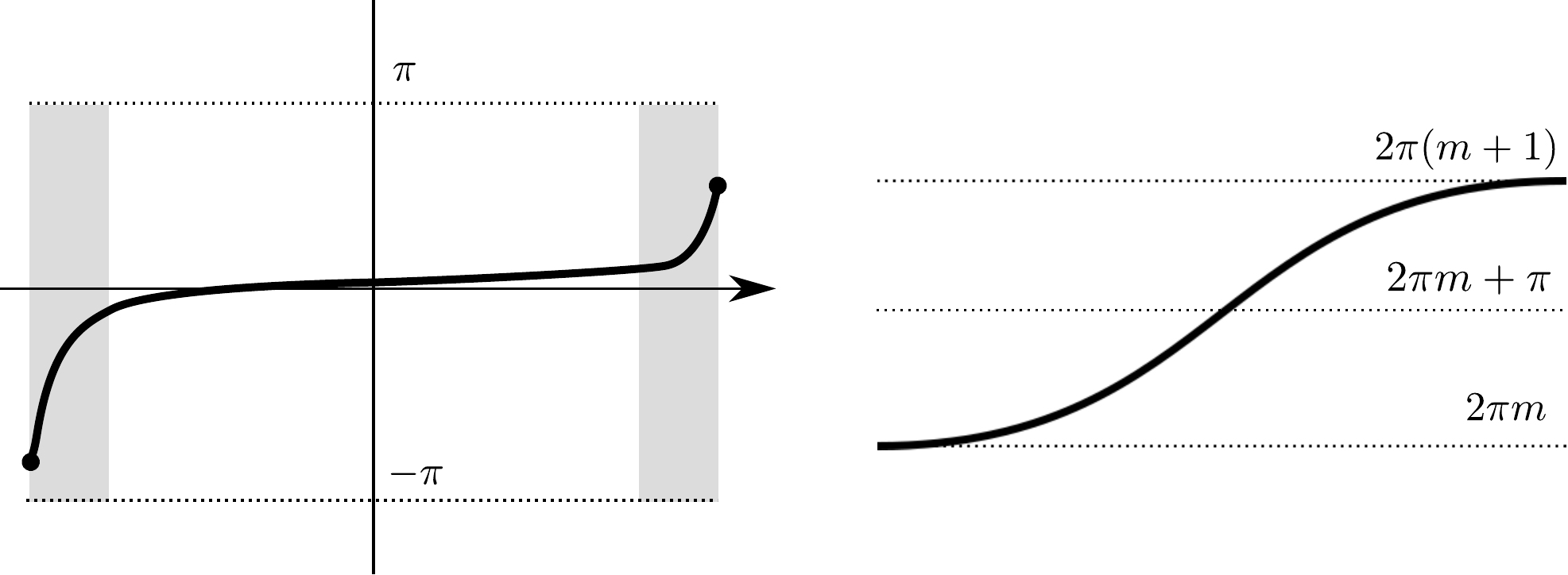}
	\caption{Tangential angle of a minimizer and a transition layer.}
	\label{figtransitionlayer3}
\end{figure}

\subsection{Related problems and remarks}

We finally mention some related problems and several remarks with comparative reviews.

The minimization of total squared curvature is studied not only in the plane but also in other manifolds or higher-dimensional spaces (e.g.\ \cite{Ko92,LaSi84,LaSi84b,LaSi85,Si08}).
In particular, there remain many open problems on elastic knots (see e.g.\ \cite{GeReVdM17,GoMaScVdM}).
The above papers basically focus on closed curves only.

Boundary value problems are rather well-studied for ``free'' elasticae, i.e., critical points of the total squared curvature without length constraint.
This problem is substantially simpler than our problem since the nontrivial configuration is essentially unique: Any nontrivial solution curve is represented by the so-called rectangular elastica up to similarity transformation.
In this problem several results including uniqueness are obtained for Dirichlet and Navier boundary value problems in e.g.\ \cite{DeGr07,DeGr09,Li93,LiJe07,Ma15}.

A free elastica is also referred as a one-dimensional Willmore surface.
Willmore surfaces are critical points of the integral of squared mean curvature in any dimension, see e.g.\ \cite{Be16,BeRi14,KeMoRi14,KuSc12,MaNe14} for recent developments, and also references therein.
Boundary value problems are also studied for Willmore surfaces (e.g.\ \cite{Da12,Ni93,Sc10}).
In particular, Willmore surfaces of revolution are studied more precisely (e.g.\ \cite{BeDaFr10,DaDeGu08,DaDeWh13,DaFrGrSc11,EiKo17,Ma17}) and a boundary layer analysis is also carried out in \cite{Gr10} under some assumptions as symmetry.
This case is more related to our problems since the corresponding equation in terms of the hyperbolic curvature is reduced to our elastica equation.

There are various other points of view even for planar curves.
For closed curves, Gage's classical result of isoperimetric inequality type \cite{Ga83} is recently generalized in \cite{BuHe15} and \cite{FeKaNi16} independently.
For open curves, the stability of post-buckling elasticae is even now a central issue (see e.g.\ \cite{GaLeMe80,JiBa15,Ma14,Sa08b,SaLe10} and references therein).
More recently, basic properties of elastic networks are investigated \cite{DaNoPl17,DaPl17}: Along the way it is confirmed in \cite{DaNoPl17} that a global minimizer in a ``drop'' setting ($l=0$ and $\theta_0$, $\theta_1$ not fixed) is the half figure-eight and thus unique up to rotation.
Free boundary problems of obstacle type are also studied in various settings; confined closed curves \cite{DoMuRo11,DaMaNo15}, graph curves above obstacles \cite{DaDe17}, and adhesion problems \cite{Ke16,Mi16,Mi16b}.
In particular, the author studied a singular limit for an adhesion problem in the paper \cite{Mi16}, from which some ideas in the present paper come.

The terminology ``phase transition'' may be confusing since phase-field methods are often used even for elastic problems (see e.g. \cite{DoLeWo17,DoMuRo11,RoSc06}).
The naive idea of phase-field methods is to approximate an objective $n$-dimensional surface by an ``interface'' of a smoothed characteristic function defined in $(n+1)$-dimension, so essentially different from our idea.

It is worth noting that our philosophy is similar to Ni and Takagi's celebrated study for a singularly perturbed elliptic equation \cite{NiTa91} (see also \cite{NiPaTa92,NiTa93}).
A common point is that they also obtain some control of least energy solutions in a singular limit, localizing the effect of energy into the boundary.

Last but not least, we do not claim that this paper is the first to point out that the borderline elastica appears near the endpoints in the straightening limit.
In fact, this has been formally indicated in Audoly and Pomeau's book in physics \cite[Section 4.4.1]{AuPo10} from a viewpoint of boundary layer analysis.
However, our result would be the first to provide a mathematical proof of the rescaled convergence, and moreover to determine the precise rate of magnification in the rescaling.

\section{Preliminaries and main results}\label{sectmainresults}

In this section we precisely describe our main results and indicate the positions of the proofs.
Throughout this paper we assume that the endpoints lie in the $x$-axis as in Figure \ref{figcurvebc}.
This loses no generality since our energies are invariant with respect to isometry.

\subsection{Extensible problem}

Let $I=(0,1)$ be the open unit interval and $\bar{I}=[0,1]$ be its closure.
For a smooth regular curve $\gamma:\bar{I}\to\mathbb{R}^2$ we denote the length by $\mathcal{L}[\gamma]$, and the total squared curvature by $\mathcal{B}[\gamma]$ as (\ref{eqntotalsquaredcurvature}).
Then, for $\varepsilon>0$, the modified total squared curvature (\ref{eqnmodifiedtotalsquaredcurvature}) is represented as $$\mathcal{E}_\varepsilon[\gamma]:=\varepsilon^2\mathcal{B}[\gamma]+\mathcal{L}[\gamma].$$
Hereafter, we use both the original parameter $t\in\bar{I}$ and the arc length parameterization $s\in[0,\mathcal{L}[\gamma]]$ as the situation demands.
For a regular curve $\gamma\in C^\infty(\bar{I};\mathbb{R}^2)$, we often denote its arc length reparameterization by $\tilde{\gamma}:[0,\mathcal{L}[\gamma]]\to\mathbb{R}^2$.

Let $l>0$ and $\theta_0,\theta_1\in[-\pi,\pi]$.
We say that a curve $\gamma\in C^\infty(\bar{I};\mathbb{R}^2)$ is {\it admissible} if $\gamma$ is regular and of constant speed, i.e., $|\dot{\gamma}|\equiv\mathcal{L}[\gamma]>0$, and moreover satisfies the clamped boundary condition:
\begin{equation}\label{eqnboundarycondition}
\begin{array}{ll}
\gamma(0)=(0,0),& \dot{\gamma}(0)=\mathcal{L}[\gamma](\cos\theta_0,\sin\theta_0), \\
\gamma(1)=(l,0),& \dot{\gamma}(1)=\mathcal{L}[\gamma](\cos\theta_1,\sin\theta_1).
\end{array}
\end{equation}
We denote the set of all admissible curves by $\mathcal{A}_{\theta_0,\theta_1,l}\subset C^\infty(\bar{I};\mathbb{R}^2)$.

For $\varepsilon>0$, we consider the following minimizing problem
\begin{eqnarray}\label{minprob1}
\min_{\gamma\in\mathcal{A}_{\theta_0,\theta_1,l}}\mathcal{E}_\varepsilon[\gamma].
\end{eqnarray}
Our object is a global minimizer, i.e., a curve $\gamma_\varepsilon$ such that $\mathcal{E}_\varepsilon[\gamma_\varepsilon]=\min_{\gamma\in \mathcal{A}_{\theta_0,\theta_1,l}}\mathcal{E}_\varepsilon[\gamma]$.
The existence of minimizers follows by a direct method in the calculus of variations and a bootstrap argument (Appendix \ref{appendixexistence}).

Our first theorem is concerned with a rescaled convergence as $\varepsilon\to0$ to a part of the borderline elastica near each endpoint as in Figure \ref{figcurvethm}.
To state the main theorem, we define borderline elasticae with initial angles as in Figure~\ref{figborderline}.

For a smooth regular curve $\gamma$ defined on an interval $\bar{J}=[0,T]$ (or $\bar{J}=[0,\infty)$) we denote by $\vartheta_\gamma$ a continuous representation of the {\em tangential angle}.
More precisely, $\vartheta_\gamma$ is a smooth function on $\bar{J}$ such that the vectors $\dot{\gamma}(t)$ and $\left(\cos\vartheta_\gamma(t),\sin\vartheta_\gamma(t)\right)$ are in a same direction for any $t\in\bar{J}$.
Such a function is unique up to the addition of constants in $2\pi\mathbb{Z}$.
Then we define borderline elasticae with initial angles as follows.

\begin{definition}[Borderline elastica with initial angle]\label{defborderline0}
	For $\theta\in[-\pi,\pi]$, we say that a smooth curve $\gamma_{B}^\theta:[0,\infty)\to\mathbb{R}^2$ parameterized by the arc length $s$ is the {\it borderline elastica with initial angle $\theta$} if
	$$\gamma_B^\theta(0)=(0,0),\quad \vartheta_{\gamma_B^\theta}(0)=\theta,\quad \lim_{s\to\infty}\vartheta_{\gamma_B^\theta}(s)=0,$$
	and moreover $|\partial_s\vartheta_{\gamma_B^\theta}|^2=1-\cos\vartheta_{\gamma_B^\theta}$ holds in $(0,\infty)$.
	Such a curve is uniquely determined for any given $\theta\in[-\pi,\pi]$.
	(See also Definition \ref{defborderline}.)
\end{definition}

\begin{figure}[tb]
	\centering
	\includegraphics[width=100mm]{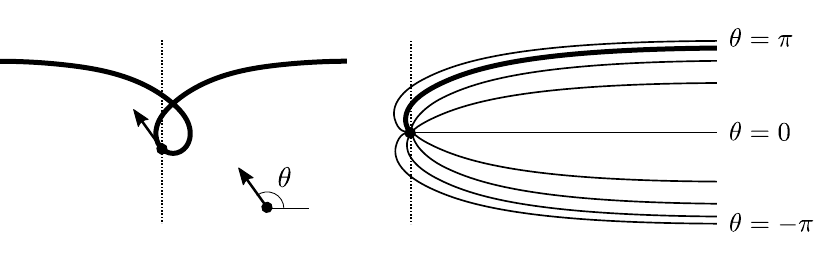}
	\caption{Borderline elastica with initial angle.}
	\label{figborderline}
\end{figure}

We are now in a position to state our first main result, which describes convergence of minimizers as $\varepsilon\to0$.

\begin{theorem}[Straightening result for extensible problem]\label{thmmain1}
	Fix any convergent sequences $l_\varepsilon\to l$ in $(0,\infty)$ and $\theta_0^\varepsilon\to\theta_0$, $\theta_1^\varepsilon\to\theta_1$ in $[-\pi,\pi]$.
	Let $\gamma_\varepsilon$ be a minimizer of $\mathcal{E}_\varepsilon$ in $\mathcal{A}_{\theta_0^\varepsilon,\theta_1^\varepsilon,l_\varepsilon}$ for $\varepsilon>0$.
	Let $\tilde{\gamma}_\varepsilon$ be the arc length parameterization of $\gamma_\varepsilon$.
	Then the following statements hold.
	\begin{enumerate}
		\item[(1)] Let $\hat{\gamma}_\varepsilon(\hat{s}):=\varepsilon^{-1}\tilde{\gamma}_\varepsilon(\varepsilon\hat{s})$.
		If $|\theta_0|<\pi$, then $\hat{\gamma}_\varepsilon$ converges to $\gamma_B^{\theta_0}$ in $C^\infty_{loc}$ as $\varepsilon\to0$.
		If $|\theta_0|=\pi$, then for any subsequence of $\{\hat{\gamma}_{\varepsilon}\}_\varepsilon$ there is a subsequence $\{\hat{\gamma}_{\varepsilon'}\}_{\varepsilon'}$ such that $\hat{\gamma}_{\varepsilon'}$ converges to $\gamma_B^{\pi}$ or $\gamma_B^{-\pi}$ in $C^\infty_{loc}$ as $\varepsilon'\to0$.
		\item[(2)] Denote the length of $\gamma_\varepsilon$ by $L_\varepsilon$.
		Let $K_{c\varepsilon}=[c\varepsilon,L_\varepsilon-c\varepsilon]$ for $c>0$.
		Then
		$$\limsup_{\varepsilon\to0}\max_{s\in K_{c\varepsilon}}|\partial_s\tilde{\gamma}_\varepsilon(s)-(1,0)|\leq4 e^{-\frac{c}{\sqrt{2}}}.$$
	\end{enumerate}
\end{theorem}

Theorem \ref{thmmain1} is proved in Section \ref{sectconvergence}.
To prove this theorem, we first prove a key step in Section \ref{sectenergyexpansion}, namely, the first order expansion of the energy minimum.
By using the expansion, in Section \ref{sectconvergence}, we first prove the rescaled convergence (1) in a weak sense, and then complete the proof of the almost straightness (2).
Finally, we improve the regularity of the rescaled convergence by using explicit expressions of the curvatures by elliptic functions.

\begin{remark}
	We give some remarks on the main theorem to clarify the meaning.
	\begin{itemize}
		\item To be more precise, the above $C^\infty_{loc}$-convergence means that for any $c>0$ the restricted rescaled curve $\hat{\gamma}_\varepsilon|_{[0,c]}$ converges to $\gamma_B^{\theta_0}|_{[0,c]}$ in $C^\infty([0,c];\mathbb{R}^2)$ as $\varepsilon\to0$.
		The rescaled curve $\hat{\gamma}_\varepsilon(\hat{s})$ is defined for $\hat{s}\in[0,L_\varepsilon/\varepsilon]$, and hence at least in $[0,l_\varepsilon/\varepsilon]$.
		Since $l_\varepsilon/\varepsilon\to\infty$, for any fixed $c>0$ there is $\varepsilon_c>0$ such that for any $\varepsilon\in(0,\varepsilon_c)$ the curve $\hat{\gamma}_\varepsilon$ is defined at least in $[0,c]$.
		Thus, the convergence of $\hat{\gamma}_\varepsilon|_{[0,c]}$ is well-defined for any $c>0$.
		\item The rescaled convergence is stated only at the origin.
		However, by symmetry, we immediately find that a similar rescaled convergence is valid for the other endpoint $(l_\varepsilon,0)$ in the following sense.
		Let $\tilde{\gamma}^*_\varepsilon$ be the backward reparameterization of half-rotated $\tilde{\gamma}_\varepsilon$ about the point $(l_\varepsilon/2,0)\in\mathbb{R}^2$.
		Let $\hat{\gamma}^*_\varepsilon(\hat{s}):=\varepsilon^{-1}\tilde{\gamma}^*_\varepsilon(\varepsilon\hat{s})$.
		Then $\hat{\gamma}^*_\varepsilon$ converges to the borderline elastica with initial angle $\theta_1$ in the same sense as (1) in Theorem \ref{thmmain1}.
		\item Theorem \ref{thmmain1} controls the whole shape of $\tilde{\gamma}_\varepsilon$ as $\varepsilon\to0$.
		Indeed, for any fixed $c>0$, a minimizer $\tilde{\gamma}_\varepsilon$ is controlled in $[0,c\varepsilon]$ by (1) and in $[c\varepsilon,L_\varepsilon-c\varepsilon]$ by (2) for any small $\varepsilon>0$.
		Moreover, by symmetry, $\tilde{\gamma}_\varepsilon$ is also controlled in $[L_\varepsilon-c\varepsilon,L_\varepsilon]$.
		\item In the case that $|\theta_0|=\pi$, the rescaled convergent limits are not unique.
		This is natural because, for example, if we additionally assume that $|\theta_0^\varepsilon|\equiv\pi$ and $|\theta^\varepsilon_1|\equiv\pi$ (or $\theta^\varepsilon_1\equiv0$), then there are two different minimizers $\gamma_\varepsilon=(x_\varepsilon,y_\varepsilon)$ and $\gamma'_\varepsilon=(x_\varepsilon,-y_\varepsilon)$.
		If $|\theta_1|\in(0,\pi)$, then there remains a possibility to obtain the uniqueness, but we then need a higher order expansion of the energy than our first order expansion.
		\item In the assumption the boundary condition is perturbed as $l_\varepsilon\to l$, $\theta_0^\varepsilon\to\theta_0$, and $\theta_1^\varepsilon\to\theta_1$.
		However, the effects do not appear in the conclusion.
		This means that our result is ``stable'' for the perturbation.
		This stability would be useful for free boundary problems as in \cite{Mi16,Mi16b}; our forthcoming paper would essentially use this stability in the study of such a free boundary problem.
		\item Theorem \ref{thmmain1} is valid only for global minimizers since there are local minimizers with loops (as in Figure~\ref{figloops}) as shown in Appendix \ref{appendixexistence}.
	\end{itemize}
\end{remark}

By using Theorem \ref{thmmain1}, we also obtain more qualitative properties of global minimizers for small $\varepsilon$.
From our viewpoint, the case of boundary angle $0$ or $\pi$ is critical in a sense.
In this paper, we often assume the following {\em generic angle condition}:
\begin{eqnarray}\label{eqngenericanglecondition}
|\theta_0|,|\theta_1|\in(0,\pi).
\end{eqnarray}

Now we give a statement on qualitative properties.
We define an {\em inflection point} of a solution curve as a point (except the endpoints) where the sign of the curvature changes.
This is well-defined since the curvature of any non-straight solution curve is represented by a nonzero elliptic function (see Proposition \ref{propelasticasolution}).
For convenience sake we define that the straight line has no inflection point.

\begin{theorem}[Qualitative properties]\label{thmqualitativeelastica}
	Fix any convergent sequences $l_\varepsilon\to l$ in $(0,\infty)$ and $\theta_0^\varepsilon\to\theta_0$, $\theta_1^\varepsilon\to\theta_1$ in $[-\pi,\pi]$.
	Then there is $\bar{\varepsilon}>0$ such that for any $\varepsilon\in(0,\bar{\varepsilon})$ any minimizer $\gamma_\varepsilon$ of $\mathcal{E}_\varepsilon$ in $\mathcal{A}_{\theta_0^\varepsilon,\theta_1^\varepsilon,l_\varepsilon}$ has no self-intersection and at most one inflection point.
	In addition, if we suppose (\ref{eqngenericanglecondition}), then the following statements hold.
	\begin{enumerate}
		\item[(1)] If $\theta_0\theta_1<0$, then there is $\bar{\varepsilon}>0$ such that for any $\varepsilon\in(0,\bar{\varepsilon})$ any minimizer $\gamma_\varepsilon$ has no inflection point, and moreover the total variation of $\vartheta_{\gamma_\varepsilon}$ is $|\theta_0^\varepsilon|+|\theta_1^\varepsilon|$.
		\item[(2)] If $\theta_0\theta_1>0$, then there is $\bar{\varepsilon}>0$ such that for any $\varepsilon\in(0,\bar{\varepsilon})$ any minimizer $\gamma_\varepsilon$ has exactly one inflection point.
		Moreover, the total variation of $\vartheta_{\gamma_\varepsilon}$ converges to $|\theta_0|+|\theta_1|$ as $\varepsilon\to0$.
	\end{enumerate}
\end{theorem}

Theorem \ref{thmqualitativeelastica} is proved in Section \ref{sectqualitativeproperty}.
This theorem roughly states that for any small $\varepsilon$ any minimizer is a straightened C-shaped or S-shaped curve as in Figure \ref{figcurvecor}.
In particular, our results also imply that for any angles such that $|\theta_0|,|\theta_1|<\pi/2$ any minimizer is represented by the graph of a function for small $\varepsilon$.

\begin{remark}\label{reminflectioncritical}
	It is more delicate to deal with the critical cases.
	In the additional part of Theorem \ref{thmqualitativeelastica}, the case of $|\theta_0|=\pi$ or $|\theta_1|=\pi$ is excluded since the sign of the curvature at the corresponding endpoint is not determined only by our convergence result.
	The case of $\theta_0\theta_1=0$ is also excluded since in this case the number may depend on how the boundary parameters converge.
	However, even in the case that $|\theta_0|>0$ and $\theta_1=0$, if we additionally assume that $\theta_1^\varepsilon\geq0$, then we can prove that any minimizer has one inflection point when $\varepsilon$ is small.
	This fact is proved in Remark \ref{reminflectioncritical2}.
	An important point is that the assumption $\theta_1^\varepsilon\geq0$ particularly includes the constant angle case that $\theta_1^\varepsilon\equiv\theta_1=0$.
\end{remark}

We finally state that, if $\theta_0\theta_1<0$ holds in the generic angle condition, then the energy $\mathcal{E}_\varepsilon$ admits a unique global minimizer for any small $\varepsilon$.
This theorem is also proved in Section \ref{sectqualitativeproperty}.

\begin{theorem}[Uniqueness]\label{thmuniqueness}
	Fix any convergent sequences $l_\varepsilon\to l$ in $(0,\infty)$ and $\theta_0^\varepsilon\to\theta_0$, $\theta_1^\varepsilon\to\theta_1$ in $[-\pi,\pi]$ with (\ref{eqngenericanglecondition}) and $\theta_0\theta_1<0$.
	Then there is $\bar{\varepsilon}>0$ such that for any $\varepsilon\in(0,\bar{\varepsilon})$ the energy $\mathcal{E}_\varepsilon$ admits a unique minimizer in $\mathcal{A}_{\theta_0^\varepsilon,\theta_1^\varepsilon,l_\varepsilon}$.
\end{theorem}

\subsection{Inextensible problem}

By using the above results, we also obtain a straightening result for the inextensible problem.
In this part we do not consider the perturbation of angles for simplicity, and concentrate our attention on changing the distance of the endpoints.

Let $0<l<L$ and $\theta_0,\theta_1\in[-\pi,\pi]$.
Let $\mathcal{A}_{\theta_0,\theta_1,l}^{L}\subset \mathcal{A}_{\theta_0,\theta_1,l}$ be the set of admissible curves $\gamma\in\mathcal{A}_{\theta_0,\theta_1,l}$ of fixed length $\mathcal{L}[\gamma]=L$.
Recall that the inextensible problem is formulated as
\begin{eqnarray}\label{minprob2}
\min_{\gamma\in\mathcal{A}_{\theta_0,\theta_1,l}^{L}}\mathcal{B}[\gamma].
\end{eqnarray}

We are concerned with the shapes of straightened elastic rods, i.e., the asymptotic shape of minimizers as the distance of the endpoints is enlarged as $l\uparrow L$ while the length $L$ and the angles $\theta_0$, $\theta_1$ are fixed.
This paper proves that in the limit $l\uparrow L$ we can rephrase (\ref{minprob2}) in terms of (\ref{minprob1}) at least in a subsequential sense.

\begin{theorem}[Straightening result for inextensible problem: general case]\label{thmmain2}
	Let $L>0$ and $\theta_0,\theta_1\in[-\pi,\pi]$ with $|\theta_0|+|\theta_1|>0$.
	Then there are sequences $l_n\uparrow L$ and $\varepsilon_n\downarrow0$ as $n\to\infty$ such that for any minimizer $\gamma_n$ of $\mathcal{B}$ in $\mathcal{A}_{\theta_0,\theta_1,l_n}^{L}$ the dilated curve $\frac{L}{l_n}\gamma_n$ is a minimizer of $\mathcal{E}_{\varepsilon_n}$ in $\mathcal{A}_{\theta_0,\theta_1,L}$, and moreover
	$$\lim_{n\to\infty}\frac{L-l_n}{\varepsilon_n}=4\sqrt{2}\left(\sin^2\frac{\theta_0}{4}+\sin^2\frac{\theta_1}{4}\right).$$
\end{theorem}

We remark that the distance of the endpoints of $\frac{L}{l_n}\gamma_n$ is fixed as $L$.
The dilation is just for the normalization to fix the endpoints of curves.
It is not effective since the magnification rate $L/l_n$ converges to $1$.

Theorem \ref{thmmain2} is proved in Section \ref{sectrelation}.
This theorem implies that similar straightening results to Theorem \ref{thmmain1} and Theorem \ref{thmqualitativeelastica} are also valid for the classical inextensible problem, at least in a subsequential straightening process.
In particular, minimizers bend at the scale $\varepsilon_n$ in a straightening process $l_n\uparrow L$.
The last equality in Theorem \ref{thmmain2} means that the leading order term of $\varepsilon_n$ is completely determined by $L-l_n$ and the angles $\theta_0$ and $\theta_1$.

\begin{remark}
	The case $\theta_0=\theta_1=0$ is quite different from others, both physically and mathematically.
	This case corresponds to buckling (Figure~\ref{figbuckling}) but not straightening.
	In addition, if $\theta_0=\theta_1=0$, then the extensible problem admits only the trivial segment minimizer, but such a segment is not admissible in the inextensible problem (except $l=L$).
	Hence, the problem (\ref{minprob2}) can not be read as (\ref{minprob1}).
\end{remark}

Theorem \ref{thmmain2} requires to take a subsequence.
It is expected to be a technical assumption, but at this time we have no proof of a full convergence for the general case.
As mentioned, the difficulty is crucially due to the lack of general theory for the uniqueness of minimizers in the extensible problem.
In fact, if a given boundary condition guarantees the uniqueness as $\varepsilon\to0$, then Theorem \ref{thmmain2} is valid in a full convergence sense:
In particular, thanks to Theorem \ref{thmuniqueness}, we are able to reach a full result providing the additional assumption that $\theta_0\theta_1<0$.

\begin{theorem}[Straightening result for inextensible problem: convex case]\label{thmmain2'}
	Let $L>0$ and $\theta_0,\theta_1\in[-\pi,\pi]$ with (\ref{eqngenericanglecondition}) and $\theta_0\theta_1<0$.
	Then there are $\bar{l}\in(0,L)$ and a strictly decreasing function $\tilde{\varepsilon}:(\bar{l},L)\to(0,\infty)$ such that for any $l\in(\bar{l},L)$, the energy $\mathcal{B}$ admits a unique minimizer in $\mathcal{A}_{\theta_0,\theta_1,l}^{L}$, the dilated curve $\frac{L}{l}\gamma_l$ is a minimizer of $\mathcal{E}_{\tilde{\varepsilon}(l)}$ in $\mathcal{A}_{\theta_0,\theta_1,L}$, and moreover
	$$\lim_{l\uparrow L}\frac{L-l}{\tilde{\varepsilon}(l)}=4\sqrt{2}\left(\sin^2\frac{\theta_0}{4}+\sin^2\frac{\theta_1}{4}\right).$$
\end{theorem}

The relation between the uniqueness and the full convergence is discussed in Section \ref{sectrelation} more precisely.
Notice that Theorem \ref{thmmain2'} and our results for the extensible problem immediately imply Theorem \ref{thmmain0}.

\section{Asymptotic expansion of the energies of minimizers}\label{sectenergyexpansion}

In this section, we prove a key step for our rescaled convergence: an asymptotic expansion of the energies of minimizers as $\varepsilon\to0$.
Throughout this section, we fix convergent sequences $l_\varepsilon\to l$ in $(0,\infty)$ and $\theta_0^\varepsilon\to\theta_0$, $\theta_1^\varepsilon\to\theta_1$ in $[-\pi,\pi]$.

\begin{lemma}\label{lemasymptitoticmodified}
	Let $\gamma_\varepsilon\in\mathcal{A}_{\theta_0^\varepsilon,\theta_1^\varepsilon,l_\varepsilon}$ be a minimizer of $\mathcal{E}_\varepsilon$ in $\mathcal{A}_{\theta_0^\varepsilon,\theta_1^\varepsilon,l_\varepsilon}$ for $\varepsilon>0$.
	Then
	$$\mathcal{E}_\varepsilon[\gamma_\varepsilon]-l_\varepsilon-8\sqrt{2}\left(\sin^2\frac{\theta_0^\varepsilon}{4}+\sin^2\frac{\theta_1^\varepsilon}{4}\right)\varepsilon=o(\varepsilon) \quad \textrm{as}\quad \varepsilon\to0.$$
\end{lemma}

In the rest of this section we prove the above lemma.
Note that it suffices to prove that, for any sequence of minimizers,
\begin{eqnarray}\label{eqnlimsupinequality}
\limsup_{\varepsilon\to0}\frac{\mathcal{E}_\varepsilon[\gamma_\varepsilon]-l_\varepsilon}{\varepsilon}\leq 8\sqrt{2}\left(\sin^2\frac{\theta_0}{4}+\sin^2\frac{\theta_1}{4}\right)
\end{eqnarray}
and
\begin{eqnarray}\label{eqnliminfinequality}
\liminf_{\varepsilon\to0}\frac{\mathcal{E}_\varepsilon[\gamma_\varepsilon]-l_\varepsilon}{\varepsilon}\geq 8\sqrt{2}\left(\sin^2\frac{\theta_0}{4}+\sin^2\frac{\theta_1}{4}\right).
\end{eqnarray}

We define an energy functional $\mathcal{F}_\varepsilon$ for any smooth regular curve $\gamma$ by
\begin{equation}\label{eqnenergyanglerepresentation}
\mathcal{F}_\varepsilon[\gamma]= \int_0^{\mathcal{L}[\gamma]}\left(\varepsilon|\partial_s\vartheta_{\tilde{\gamma}}|^2+\frac{1}{\varepsilon}(1-\cos\vartheta_{\tilde{\gamma}})\right) ds,
\end{equation}
where $\theta_{\tilde{\gamma}}$ is the tangential angle of the arc length parameterization $\tilde{\gamma}$ of $\gamma$.
Note that $\mathcal{F}$ is well-defined since this energy is invariant by the addition of constants of $2\pi\mathbb{Z}$ to $\vartheta_{\tilde{\gamma}}$.
Moreover, we notice that for any $\gamma\in\mathcal{A}_{\theta_0^\varepsilon,\theta_1^\varepsilon,l_\varepsilon}$ the relation
$$\mathcal{F}_\varepsilon[\gamma]=\frac{\mathcal{E}_\varepsilon[\gamma]-l_\varepsilon}{\varepsilon}$$
holds since
$$\int_0^{\mathcal{L}[\gamma]}|\partial_s\vartheta_{\tilde{\gamma}}|^2ds=\int_{\gamma}\kappa^2ds,\quad \int_0^{\mathcal{L}[\gamma]}ds=\int_{\gamma}ds, \quad \int_0^{\mathcal{L}[\gamma]}\cos\vartheta_{\tilde{\gamma}} ds=l_\varepsilon.$$
As mentioned in Introduction, the representation $\mathcal{F}_\varepsilon$ is essentially used in this paper.

The following lemma is obvious by definition but frequently used.

\begin{lemma}\label{lemenergydecomposition}
	Let $N$ be a positive integer and $t_0<\dots<t_N$ be real numbers.
	Let $\bar{J}=[t_0,t_N]$ and $\bar{J}_i=[t_{i},t_{i+1}]$ for $i=0,\dots,N-1$.
	For any $\varepsilon>0$ and any smooth constant speed curve $\gamma:\bar{J}\to\mathbb{R}^2$,
	$$\mathcal{F}_\varepsilon[\gamma]=\sum_{i=0}^{N-1}\mathcal{F}_\varepsilon[\gamma|_{\bar{J}_i}]$$
	and each term of the right-hand sum is nonnegative.
	In particular, $\mathcal{F}_\varepsilon[\gamma]\geq \mathcal{F}_\varepsilon[\gamma|_{\bar{J}_i}]$ holds for each $i$.
\end{lemma}

\subsection{Weighted total variation}

The following weighted variation function is also frequently used in this paper.

\begin{definition}[Weighted variation of tangential angle]
	Define a strictly increasing function $V\in C^1(\mathbb{R})$ by
	$$V(\theta):=\int_0^\theta 2\sqrt{1-\cos\phi}d\phi.$$
\end{definition}

\begin{remark}[Calculation of weighted variation]\label{remvariation}
	By the half-angle formula, for any $\theta\in[-\pi,\pi]$ we calculate
	$$V(\theta)=\textrm{sign}(\theta) \cdot 8\sqrt{2}\sin^2\frac{\theta}{4}.$$
	By the periodicity, for any $m\in\mathbb{Z}$ and $\theta\in[(2m-1)\pi,(2m+1)\pi)$ we have
	$$V(\theta)=\textrm{sign}(\llbracket\theta\rrbracket) \cdot 8\sqrt{2}\sin^2\frac{\llbracket\theta\rrbracket}{4} + 8\sqrt{2}m,$$
	where $\llbracket\theta\rrbracket$ denotes a unique angle in $[-\pi,\pi)$ so that $\theta-\llbracket\theta\rrbracket\in 2\pi\mathbb{Z}$.
	Hereafter, we frequently use the notation $\llbracket\cdot\rrbracket$ in this sense.
\end{remark}

The weighted variation is essential for our arguments since the following lower estimate holds.

\begin{lemma}\label{lemlowerbound}
	For any $\varepsilon>0$ and smooth regular curve $\gamma$ parameterized by the arc length $s$, we have
	\begin{eqnarray}\nonumber
	\mathcal{F}_\varepsilon[\gamma] \geq \int_0^{\mathcal{L}[\gamma]}|\partial_
	s(V\circ\vartheta_\gamma)|ds \geq \left|V\left(\vartheta_\gamma(\mathcal{L}[\gamma])\right)-V\left(\vartheta_\gamma(0)\right)\right|.
	\end{eqnarray}
\end{lemma}

\begin{proof}
	The first inequality follows by the definition of $\mathcal{F}$ and the inequality $\varepsilon X^2+\varepsilon^{-1}Y^2\geq2|X||Y|$.
	The last inequality follows by the triangle inequality.
\end{proof}

To compute the above lower bound, the following lemma is useful.

\begin{lemma}\label{lemvartheta}
	Let $\theta,\theta'\in\mathbb{R}$.
	Then the following inequality holds:
	$$|V(\theta)-V(\theta')| \geq 8\sqrt{2}\left|\sin^2\frac{\llbracket\theta\rrbracket}{4}-\sin^2\frac{\llbracket\theta'\rrbracket}{4}\right|.$$
	The equality is attained if and only if $\theta,\theta'\in[m\pi,(m+1)\pi]$ for some $m\in\mathbb{Z}$.
\end{lemma}

\begin{proof}
	Fix $\theta,\theta'\in\mathbb{R}$.
	Then there exists $\theta^*\in\mathbb{R}$ with $\theta^*\leq\theta$ so that $|\llbracket\theta^*\rrbracket|=|\llbracket\theta'\rrbracket|$ and $\theta^*,\theta\in[m\pi,(m+1)\pi]$ for some $m\in\mathbb{Z}$.
	By periodicity, we have $|\theta-\theta'|\geq|\theta-\theta^*|$, and hence
	$$|V(\theta)-V(\theta')|\geq|V(\theta)-V(\theta^*)|.$$
	By Remark \ref{remvariation}, the right-hand term is calculated as
	$$|V(\theta)-V(\theta^*)| = 8\sqrt{2}\left|\sin^2\frac{\llbracket\theta\rrbracket}{4}-\sin^2\frac{\llbracket\theta^*\rrbracket}{4}\right|.$$
	Since $\sin^2(\llbracket\theta^*\rrbracket/4)=\sin^2(\llbracket\theta'\rrbracket/4)$, the desired inequality holds.
	In view of the first inequality, the equality is attained if and only if $\theta,\theta'\in[m\pi,(m+1)\pi]$ for some $m\in\mathbb{Z}$ from the beginning.
	The proof is complete.
\end{proof}

\subsection{Lower bound for the modified squared curvature}

In this subsection we prove the liminf inequality (\ref{eqnliminfinequality}), that is, the following proposition.

\begin{proposition}\label{propliminfinequality}
	Let $\gamma_\varepsilon\in\mathcal{A}_{\theta_0^\varepsilon,\theta_1^\varepsilon,l_\varepsilon}$ be a minimizer of $\mathcal{E}_\varepsilon$ in $\mathcal{A}_{\theta_0^\varepsilon,\theta_1^\varepsilon,l_\varepsilon}$ for $\varepsilon>0$.
	Then
	$$\liminf_{\varepsilon\to0}\mathcal{F}_\varepsilon[\gamma_\varepsilon]\geq 8\sqrt{2}\left(\sin^2\frac{\theta_0}{4}+\sin^2\frac{\theta_1}{4}\right).$$
\end{proposition}

We first confirm basic convergences on a sequence of minimizers.

\begin{proposition}\label{propsegmentconvergence}
	Let $\gamma_\varepsilon\in\mathcal{A}_{\theta_0^\varepsilon,\theta_1^\varepsilon,l_\varepsilon}$ be a minimizer of $\mathcal{E}_\varepsilon$ in $\mathcal{A}_{\theta_0^\varepsilon,\theta_1^\varepsilon,l_\varepsilon}$ for $\varepsilon>0$.
	Then the length $L_\varepsilon$ of $\gamma_\varepsilon$ converges to $l$, and the curve $\gamma_\varepsilon$ uniformly converges to the segment $\bar{\gamma}(t)=(lt,0)$, $t\in\bar{I}$, as $\varepsilon\to0$.
\end{proposition}

\begin{proof}
	Let $L_\varepsilon=\mathcal{L}[\gamma_\varepsilon]$ be the length (speed) of $\gamma_\varepsilon$.
	It is easy to confirm that $\mathcal{E}_\varepsilon[\gamma_\varepsilon]\to l$ as $\varepsilon\to0$ since we can easily construct a sequence of curves $\gamma'_\varepsilon\in\mathcal{A}_{\theta_0^\varepsilon,\theta_1^\varepsilon,l_\varepsilon}$ such that $\mathcal{E}_\varepsilon[\gamma'_\varepsilon]\to l$ by using circular arcs of radius $\varepsilon$ and a segment.
	Since $l_\varepsilon\leq L_\varepsilon\leq \mathcal{E}_\varepsilon[\gamma_\varepsilon]$ and $l_\varepsilon\to l$, the length (speed) $L_\varepsilon$ also converges to $l$.
	In addition, since the speeds $L_\varepsilon$ are bounded as $\varepsilon\to0$, the curves $\gamma_\varepsilon$ are equicontinuous as $\varepsilon\to0$.
	Moreover, since the endpoint $\gamma_\varepsilon(0)=(0,0)$ is fixed and the lengths are bounded, we also find that the curves $\gamma_\varepsilon$ are uniformly bounded as $\varepsilon\to0$.
	Thus, by the Arzel\`{a}-Ascoli theorem, up to a subsequence of any subsequence, $\gamma_\varepsilon$ uniformly converges to a continuous curve joining $(0,0)$ to $(l,0)$.
	Since $L_\varepsilon\to l$ and $\gamma_\varepsilon$ is of constant speed, the convergent limit must be the segment $\bar{\gamma}$.
	Hence, $\gamma_\varepsilon$ fully converges to the segment $\bar{\gamma}$.
	The proof is complete.
\end{proof}

For such a convergent sequence, the following elementary lemma holds.

\begin{lemma}\label{lemmeanvalue}
	Let $l>0$.
	Suppose that a sequence of smooth constant speed curves $\gamma_\varepsilon$ uniformly converges to the segment $\bar{\gamma}(t)=(lt,0)$, and moreover the length $L_\varepsilon$ of $\gamma_{\varepsilon}$ converges to $l$ as $\varepsilon\to0$.
	Then for any open subinterval $J\subset I$ there is a sequence of times $\{t_\varepsilon\}_\varepsilon\subset J$ such that $\llbracket \vartheta_{\gamma_\varepsilon}(t_\varepsilon) \rrbracket\to0$ as $\varepsilon\to0$.
\end{lemma}

\begin{proof}
	We prove by contradiction; suppose that there is an open interval $J\subset I$ such that $\inf_J|\llbracket \vartheta_{\gamma_\varepsilon}\rrbracket|$ does not converge to $0$ as $\varepsilon\to0$, i.e., there are $\delta>0$ and a sequence $\varepsilon_j\to0$ such that $\inf_J|\llbracket \vartheta_{\gamma_{\varepsilon_j}}\rrbracket|\geq\delta$ for any $j$.
	By this assumption, the $x$-component of $\gamma_{\varepsilon_j}$ satisfies
	$$\limsup_{j\to\infty}(x_{\varepsilon_j}(t_1)-x
	_{\varepsilon_j}(t_0))= \limsup_{j\to\infty}L_{\varepsilon_j}\int_{J}\cos\vartheta_{\gamma_{\varepsilon_j}}dt \leq l(t_1-t_0)(\cos\delta)< l(t_1-t_0),$$
	where the convergence $L_{\varepsilon_j}\to l$ is used.
	On the other hand, since $\gamma_{\varepsilon_j}$ converges to the segment $\bar{\gamma}(t)=(lt,0)$, we immediately have
	$$\lim_{j\to\infty}(x_{\varepsilon_j}(t_1)-x
	_{\varepsilon_j}(t_0))= l(t_1-t_0).$$
	This is a contradiction.
\end{proof}

\begin{remark}
	The above lemma is elementary but should be slightly noted, since there is an example of a sequence of curves such that the sequence uniformly converges to a segment but the tangent vectors are uniformly away from the rightward vector anywhere.
	Such an example is constructed as in Figure~\ref{fignonrightward}, namely, as ``sawtooth'' curves of which edges are modified by loops, so that the number of the tooths diverges and the loops rapidly degenerate to points in the limit.
	Hence, the length convergence is an essential assumption.
\end{remark}

\begin{figure}[tb]
	\centering
	\includegraphics[width=60mm]{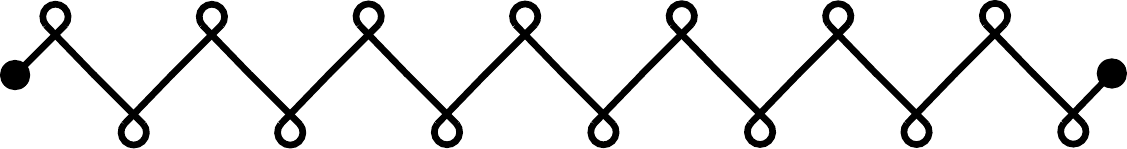}
	\caption{An example of a curve of which tangent vector is not rightward.}
	\label{fignonrightward}
\end{figure}

We are now in a position to prove Proposition \ref{propliminfinequality}.

\begin{proof}[Proof of Proposition \ref{propliminfinequality}]
	By Proposition \ref{propsegmentconvergence} and Lemma \ref{lemmeanvalue}, for $\varepsilon>0$ there is $t_{\varepsilon}\in I$ such that $\llbracket\vartheta_{\gamma_\varepsilon}(t_{\varepsilon})\rrbracket\to0$ as $\varepsilon\to0$.
	Then, by Lemma \ref{lemlowerbound} and Lemma \ref{lemvartheta}, we find that
	\begin{eqnarray}
	\mathcal{F}_{\varepsilon}[\gamma_{\varepsilon}] &=& \mathcal{F}_{\varepsilon}[\gamma_{\varepsilon}|_{[0,t_{\varepsilon}]}]+\mathcal{F}_{\varepsilon}[\gamma_{\varepsilon}|_{[t_{\varepsilon},1]}] \nonumber\\
	&\geq& 8\sqrt{2}\left|\sin^2\frac{\llbracket\vartheta_{\gamma_{\varepsilon}}(0)\rrbracket}{4}-\sin^2\frac{\llbracket\vartheta_{\gamma_{\varepsilon}}(t_{\varepsilon})\rrbracket}{4}\right|\nonumber\\
	&&+8\sqrt{2}\left|\sin^2\frac{\llbracket\vartheta_{\gamma_{\varepsilon}}(t_{\varepsilon})\rrbracket}{4}-\sin^2\frac{\llbracket\vartheta_{\gamma_{\varepsilon}}(1)\rrbracket}{4}\right|. \nonumber
	\end{eqnarray}
	Since $$\sin^2\frac{\llbracket\vartheta_{\gamma_{\varepsilon}}(0)\rrbracket}{4}=\sin^2\frac{\theta^{\varepsilon}_0}{4},\quad \sin^2\frac{\llbracket\vartheta_{\gamma_{\varepsilon}}(1)\rrbracket}{4}=\sin^2\frac{\theta^{\varepsilon}_1}{4},$$
	and the convergences $\theta^{\varepsilon}_0\to\theta_0$, $\theta^{\varepsilon}_1\to\theta_1$, $\llbracket\vartheta_{\gamma_\varepsilon}(t_{\varepsilon})\rrbracket\to0$ hold as $\varepsilon\to0$, we obtain
	$$\liminf_{\varepsilon\to0}\mathcal{F}_{\varepsilon}[\gamma_{\varepsilon}]\geq8\sqrt{2}\left(\sin^2\frac{\theta_0}{4}+\sin^2\frac{\theta_1}{4}\right).$$
	The proof is complete.
\end{proof}

\subsection{Construction of curves with energy convergence}\label{subsectupperbound}

In this subsection we prove that the limsup inequality (\ref{eqnlimsupinequality}) holds for any sequence of minimizers.
To this end, it suffices to construct a suitable sequence of test curves so that the energies converge to the right-hand term of (\ref{eqnlimsupinequality}).

\begin{proposition}\label{proplimsupinequality}
	There is a sequence of curves $\gamma'_\varepsilon\in\mathcal{A}_{\theta_0^\varepsilon,\theta_1^\varepsilon,l_\varepsilon}$ such that
	\begin{eqnarray}\label{eqnenergylimit}
	\lim_{\varepsilon\to0}\mathcal{F}_\varepsilon[\gamma'_\varepsilon]=8\sqrt{2}\left(\sin^2\frac{\theta_0}{4}+\sin^2\frac{\theta_1}{4}\right).
	\end{eqnarray}
\end{proposition}

This immediately implies (\ref{eqnlimsupinequality}) for any sequence of minimizers $\{\gamma_\varepsilon\}_\varepsilon$ since $\mathcal{F}_\varepsilon[\gamma_\varepsilon]$ is bounded above by $F_\varepsilon[\gamma'_\varepsilon]$ for a curve $\gamma'_\varepsilon$ in Proposition \ref{proplimsupinequality}.
For the proof, we construct suitable curves which are ``optimally bending'' as $\varepsilon\to0$ near the endpoints.
Some ideas are similar to the author's previous paper \cite{Mi16}.

In view of phase transitions, near the endpoints, the rescaled tangential angles are expected to be close to transition layers for the phase transition energy $\mathcal{F}_\varepsilon$.
Hence, we consider the following ODEs:
\begin{eqnarray}\label{eqnborderline}
\partial_s\varphi_+(s)=\sqrt{1-\cos\varphi_+(s)}, \quad \partial_s\varphi_-(s)=-\sqrt{1-\cos\varphi_-(s)}.
\end{eqnarray}

For any initial values $\varphi_\pm(0)\in\mathbb{R}$, these equations are solved uniquely and globally in $s\in\mathbb{R}$.
When $\varphi_\pm(0)\in2\pi\mathbb{Z}$, the solutions are constant functions.
In the case that $\varphi_\pm(0)=\pm\pi$, the solutions are represented as
\begin{eqnarray}\label{solborderline}
\bar{\varphi}_\pm(s):=\pm 4\arctan\left(e^{\frac{s}{\sqrt{2}}}\right).
\end{eqnarray}
The function $\bar{\varphi}_+$ is strictly increasing with $\lim_{s\to\pm\infty}\bar{\varphi}_+(s)=\pi\pm\pi$ and its graph possesses point symmetry at $(0,\bar{\varphi}_+(0))=(0,\pi)$.
Any other solution to (\ref{eqnborderline}) is of the form $\bar{\varphi}_{\pm}(s+s_0)+2\pi m$, where $s_0\in\mathbb{R}$ and $m\in\mathbb{Z}$.

An important property of the above solutions is that for any $s_0<s_1$, by (\ref{eqnborderline}), the following energy identity holds:
\begin{eqnarray}\label{eqnenergy}
\int_{s_0}^{s_1}\left(|\partial_s\bar{\varphi}_\pm|^2 + (1-\cos\bar{\varphi}_\pm) \right)ds &=& \pm\int_{s_0}^{s_1} 2\partial_s\bar{\varphi}_\pm\sqrt{1-\cos\bar{\varphi}_\pm}ds \\
&=& \pm\int_{s_0}^{s_1} \partial_s(V\circ\bar{\varphi}_\pm) ds \nonumber\\
&=& \pm(V\circ\bar{\varphi}_\pm(s_1)-V\circ\bar{\varphi}_\pm(s_0)), \nonumber\\
&=& |V(\bar{\varphi}_\pm(s_1))-V(\bar{\varphi}_\pm(s_0))|,\nonumber
\end{eqnarray}
where $V$ is the weighted variation function.
The last equality follows since $V$ is increasing and $\bar{\varphi}_+$ (resp.\ $\bar{\varphi}_-$) is increasing (resp.\ decreasing).

A non-straight unit speed curve of which tangential angle satisfies (\ref{eqnborderline}) is nothing but the borderline elastica;
in fact, concerning (\ref{solborderline}) for example, we compute the curvature as
$$\bar{\kappa}_\pm(s)=\partial_s\bar{\varphi}_\pm(s)=\pm\sqrt{2}{\rm \ sech}\frac{s}{\sqrt{2}}.$$
(See e.g.\ \cite{Si08} to confirm that the above expression corresponds to the borderline elastica.)
By (\ref{eqnborderline}) and (\ref{solborderline}), the borderline elasticae $\bar{\gamma}_\pm=(\bar{x}_\pm,\bar{y}_\pm)$ such that $\bar{\gamma}_\pm(0)=(0,0)$ and $\partial_s\bar{\gamma}_\pm(0)=(-1,0)$ are explicitly parameterized as
\begin{eqnarray}
\bar{x}_\pm(s) &=& \int_0^s\cos\bar{\varphi}_\pm=s-\int_0^s(1-\cos\bar{\varphi}_\pm) = s\mp\int_0^s\partial_s\bar{\varphi}_\pm\sqrt{1-\cos\bar{\varphi}_\pm} \nonumber\\
&=& s-\sqrt{2}\int_0^s\partial_s\bar{\varphi}_\pm\sin\frac{\bar{\varphi}_\pm}{2} = s + 2\sqrt{2}\cos\frac{\bar{\varphi}_\pm(s)}{2}=s - 2\sqrt{2}\tanh\frac{s}{\sqrt{2}}, \nonumber\\
\bar{y}_\pm(s) &=& \int_0^s\sin\bar{\varphi}_\pm=\mp\int_0^{|s|}\sqrt{1-\cos^2\bar{\varphi}_\pm} = -\int_0^{|s|}\partial_s\bar{\varphi}_\pm\sqrt{1+\cos\bar{\varphi}_\pm} \nonumber\\
&=& \sqrt{2}\int_0^{|s|}\partial_s\bar{\varphi}_\pm\cos\frac{\bar{\varphi}_\pm}{2} = 2\sqrt{2}\left(\sin\frac{\bar{\varphi}_\pm(s)}{2} \mp 1 \right)= \pm 2\sqrt{2}\left({\rm sech}\frac{s}{\sqrt{2}}-1\right). \nonumber
\end{eqnarray}

Using the borderline elasticae, we can construct a sequence of curves satisfying (\ref{eqnenergylimit}).
For the sake of convenience, we prepare a precise definition of borderline elasticae, which is equivalent to Definition \ref{defborderline0}.

\begin{definition}[Borderline elastica with initial angle]\label{defborderline}
	Let $\theta\in[-\pi,\pi]$.
	A function $\vartheta_B^\theta:[0,\infty)\to\mathbb{R}$ is called {\it borderline angle function with initial angle $\theta$} if $\vartheta_B^\theta$ is a solution to either of the equations (\ref{eqnborderline}) such that $\vartheta_B^\theta(0)=\theta$ and $\vartheta_B^\theta(s)\to0$ as $s\to\infty$.
	Such a function is uniquely determined for any $\theta\in[-\pi,\pi]$.

	Similarly, a smooth curve $\gamma_B^\theta:[0,\infty)\to\mathbb{R}^2$ parameterized by the arc length is called {\it borderline elastica with initial angle $\theta$} if $\gamma_B^\theta(0)=(0,0)$ and its tangential angle $\vartheta_{\gamma_B^\theta}$ is the borderline angle function with initial angle $\theta$ in the above sense.
\end{definition}

Now we construct a bending part near the origin.

\begin{lemma}\label{lemcurveconstruction}
	Let $\alpha\in(0,1)$ and $\theta_\varepsilon\to\theta$ be a convergent sequence in $[-\pi,\pi]$.
	Then there is a sequence of smooth regular curves $\gamma_\varepsilon=(x_\varepsilon,y_\varepsilon)$ parameterized by the arc lengths $s\in[0,\varepsilon^\alpha]$ such that the following conditions hold:
	\begin{enumerate}
		\item[(1)] $\gamma_\varepsilon(0)=(0,0)$, $-2\sqrt{2}\varepsilon\leq x_\varepsilon(s)\leq\varepsilon^{\alpha}$ and $|y_\varepsilon(s)|\leq 2\sqrt{2}\varepsilon$ for $s\in[0,\varepsilon^\alpha]$.
		\item[(2)] $\vartheta_{\gamma_\varepsilon}(0)=\theta_{\varepsilon}$ and $\lim_{\varepsilon\to0}\vartheta_{\gamma_\varepsilon}(\varepsilon^{\alpha})=0$.
		\item[(3)] $\lim_{\varepsilon\to0}\mathcal{F}_\varepsilon[\gamma_\varepsilon]=8\sqrt{2}\sin^2({\theta}/{4}).$
	\end{enumerate}
\end{lemma}

\begin{proof}
	We prove this lemma by using a part of the rescaled borderline elastica:
	we define the curve $\gamma_\varepsilon$ so that $\mathcal{L}[\gamma_\varepsilon]=\varepsilon^\alpha$ and $\gamma_\varepsilon(s)=\varepsilon\gamma_B^{\theta_\varepsilon}(s/\varepsilon)$ for $s\in(0,\varepsilon^\alpha)$, where $\gamma_B^{\theta_\varepsilon}$ is the borderline elastica with initial angle ${\theta_\varepsilon}$ in Definition \ref{defborderline}.
	Note that $\vartheta_{\gamma_\varepsilon}(s)=\vartheta_B^{\theta_\varepsilon}(s/\varepsilon)$.
	By the aforementioned properties of the borderline elastica, it is straightforward to confirm the conditions (1) and (2).
	It should be noted that $\varepsilon^{\alpha-1}\to\infty$ as $\varepsilon\to0$ by $\alpha<1$, and hence $\vartheta_{\gamma_{\varepsilon}}(\varepsilon^\alpha)=\vartheta_B^{\theta_\varepsilon}(\varepsilon^{\alpha-1})$ converges to zero as $\varepsilon\to0$.
	The last condition (3) follows by the energy identity (\ref{eqnenergy}):
	\begin{eqnarray}
	\mathcal{F}_\varepsilon[\gamma_\varepsilon] &=& \int_0^{\varepsilon^\alpha} \left(\varepsilon|\partial_s\vartheta_{\gamma_\varepsilon}|^2 + \frac{1}{\varepsilon}(1-\cos\vartheta_{\gamma_\varepsilon}) \right) ds \nonumber\\
	&=& \int_0^{\varepsilon^{\alpha-1}} \left(|\partial_{s'}\vartheta_B^{\theta_\varepsilon}|^2 + (1-\cos{\vartheta_B^{\theta_\varepsilon}})\right) ds' \quad(s'=s/\varepsilon) \nonumber\\
	&=& |V(\vartheta_B^{\theta_\varepsilon}(0))-V(\vartheta_B^{\theta_\varepsilon}(\varepsilon^{\alpha-1}))|\xrightarrow{\varepsilon\to0}|V({\theta}) -V(0)|=8\sqrt{2}\sin^2\frac{{\theta}}{4}, \nonumber
	\end{eqnarray}
	where Lemma \ref{lemvartheta} is used for the last identity.
\end{proof}

We next construct a suitable sequence of curves connecting the parts near the endpoints.

\begin{lemma}\label{lemcurveconstruction2}
	Let $A_\varepsilon=(a_\varepsilon^x,a_\varepsilon^y),B_\varepsilon=(b_\varepsilon^x,b_\varepsilon^y)\in\mathbb{R}^2$ be points such that $A_\varepsilon\to(0,0)$ and $B_\varepsilon\to(l,0)$ as $\varepsilon\to0$ for some $l>0$.
	Let $\theta_\varepsilon^A,\theta_\varepsilon^B\in[-\pi,\pi]$ be angles converging to zero as $\varepsilon\to0$.
	Suppose that $|a_\varepsilon^y|+|b_\varepsilon^y|=o(\varepsilon^{1/2})$ as $\varepsilon\to0$.
	Then there is a sequence of smooth curves $\gamma_\varepsilon$ of length $L_\varepsilon$ parameterized by the arc lengths $s\in[0,L_\varepsilon]$ such that the boundary condition
	$$\gamma_\varepsilon(0)=A_\varepsilon,\ \gamma(L_\varepsilon)=B_\varepsilon,\ \partial_s\gamma_\varepsilon(0)=(\cos\theta_\varepsilon^A,\sin\theta_\varepsilon^A),\ \partial_s\gamma_\varepsilon(L_\varepsilon)=(\cos\theta_\varepsilon^B,\sin\theta_\varepsilon^B)$$
	hold, the length $L_\varepsilon$ converges to $l$, and moreover
	$$\lim_{\varepsilon\to0}\mathcal{F}_\varepsilon[\gamma_\varepsilon]=0.$$
\end{lemma}

\begin{proof}
	We first note that it suffices to construct a sequence of curves of class $C^1$ and piecewise $C^2$ by a standard mollifying argument.
	We construct $\gamma_\varepsilon$ as in Figure~\ref{figupperbound}; namely, we use circular arcs of radius $\varepsilon$ near the endpoints, and connect them by a segment.

	By using circular arcs of radius $\varepsilon$ and central angles $\phi_\varepsilon^A$, $\phi_\varepsilon^B$ such that $\phi_\varepsilon^A,\phi_\varepsilon^B\to0$ near the endpoints (and noting Lemma \ref{lemenergydecomposition}), we can modify the boundary conditions as $A'_\varepsilon,B'_\varepsilon,\theta_\varepsilon^{A'},\theta_\varepsilon^{B'}$ such that $A'_\varepsilon,B'_\varepsilon$ satisfy the same assumptions as $A_\varepsilon,B_\varepsilon$, and $\theta_\varepsilon^{A'}=\theta_\varepsilon^{B'}=0$ for any small $\varepsilon>0$.
	Note that the energy $\mathcal{F}_\varepsilon$ of the circular arc parts $\gamma_\varepsilon^c$ tends to be zero as $\varepsilon\to0$ since
	$$\varepsilon\int_{\gamma_\varepsilon^c}\kappa^2ds=\varepsilon\cdot\frac{1}{\varepsilon^2}\cdot\varepsilon(\phi_\varepsilon^A+\phi_\varepsilon^B)\to0,\quad \frac{1}{\varepsilon}\int_{\gamma_\varepsilon^c}ds=\frac{1}{\varepsilon}\cdot\varepsilon(\phi_\varepsilon^A+\phi_\varepsilon^B)\to0,$$
	$$\frac{1}{\varepsilon}\left|\int_{\gamma_\varepsilon^c}\cos\vartheta_{\gamma_\varepsilon^c}ds\right|\leq \frac{1}{\varepsilon}\int_{\gamma_\varepsilon^c}ds \to 0.$$

	Then, by using again circular arcs of radius $\varepsilon$ such that the central angles converge to zero, we may assume that the boundary condition $A''_\varepsilon,B''_\varepsilon,\theta_\varepsilon^{A''},\theta_\varepsilon^{B''}$ allow a segment that is compatible with the condition.

	The energy $\mathcal{F}_\varepsilon$ of the segment $\gamma_\varepsilon^s$ joining $A''_\varepsilon=({a_\varepsilon^x}'',{a_\varepsilon^y}'')$ to $B''_\varepsilon=({b_\varepsilon^x}'',{b_\varepsilon^y}'')$ also satisfies $\mathcal{F}_\varepsilon[\gamma_\varepsilon^s]\to0$.
	In fact, the curvature of the segment is zero, and
	\begin{eqnarray}
	\frac{1}{\varepsilon}\int_{\gamma_\varepsilon^s}(1-\cos\vartheta_{\gamma_\varepsilon^s})ds &=& \frac{1}{\varepsilon}\left(\sqrt{|{a_\varepsilon^x}''-{b_\varepsilon^x}''|^2+|{a_\varepsilon^y}''-{b_\varepsilon^y}''|^2}-|{a_\varepsilon^x}''-{b_\varepsilon^x}''|\right)\nonumber\\
	&=& \varepsilon^{-1} o(|{a_\varepsilon^y}''-{b_\varepsilon^y}''|^2)=o(1)\to0 \nonumber
	\end{eqnarray}
	since $|{a_\varepsilon^x}''-{b_\varepsilon^x}''|\to l>0$ and $|{a_\varepsilon^y}''|+|{b_\varepsilon^y}''|=o(\varepsilon^{1/2})$.
	The proof is now complete.
\end{proof}

\begin{figure}[tb]
	\centering
	\includegraphics[width=80mm]{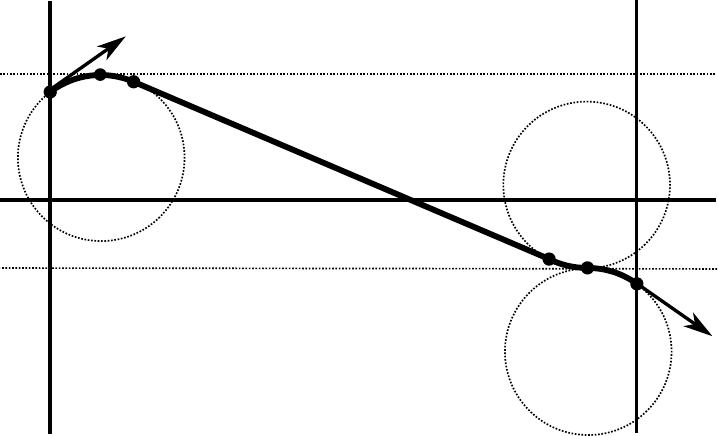}
	\caption{Construction of a curve for Lemma \ref{lemcurveconstruction2}.}
	\label{figupperbound}
\end{figure}

By using the above lemmas, we complete the proof of Proposition \ref{proplimsupinequality}.

\begin{proof}[Proof of Proposition \ref{proplimsupinequality}]
	As mentioned in the proof of Lemma \ref{lemcurveconstruction2}, it suffices to construct a sequence of curves of class $C^1$ and piecewise $C^2$ by a standard mollifying argument.
	We construct a sequence $\{\gamma'_\varepsilon\}_\varepsilon$ as in Figure~\ref{figupperbound2}.

	Fix any $\alpha\in(0,1)$.
	Let $\varepsilon$ be small as $\varepsilon^\alpha<l_\varepsilon$.
	To construct $\gamma'_\varepsilon$, we use the curves in Lemma \ref{lemcurveconstruction} near the endpoints and connect them suitably by Lemma \ref{lemcurveconstruction2}.
	Namely, denoting the curves of Lemma \ref{lemcurveconstruction} with $\theta=\theta^\varepsilon_i$ ($i=0,1$) by $\gamma^i_\varepsilon$, we define $\gamma_\varepsilon'$ so that
	\begin{eqnarray}\nonumber
	\gamma_\varepsilon'(s)=
	\begin{cases}
	\gamma_\varepsilon^0(s), & s\in[0,\varepsilon^\alpha],\\
	\gamma_\varepsilon''(s-\varepsilon^\alpha), & s\in[\varepsilon^\alpha,L'_\varepsilon-\varepsilon^\alpha],\\
	(l_\varepsilon,0)-\gamma_\varepsilon^1(L'_\varepsilon-s), & s\in[L'_\varepsilon-\varepsilon^\alpha,L'_\varepsilon],
	\end{cases}
	\end{eqnarray}
	where the connecting part $\gamma_\varepsilon''$ is taken as in Lemma \ref{lemcurveconstruction2} of which boundary condition is suitably set so that $\gamma'_\varepsilon$ is of class $C^1$ (the length $L'_\varepsilon$ is a posteriori defined).
	Note that the points and tangential angles at $s=\varepsilon^\alpha$ and $s=L'_\varepsilon-\varepsilon^\alpha$ satisfy the assumptions in Lemma \ref{lemcurveconstruction2} by Lemma \ref{lemcurveconstruction}.
	Then, since Lemma \ref{lemenergydecomposition} implies that
	$$\mathcal{F}_{\varepsilon}[\gamma_\varepsilon']=\mathcal{F}_{\varepsilon}[\gamma_\varepsilon'|_{[0,\varepsilon^\alpha]}]+\mathcal{F}_{\varepsilon}[\gamma_\varepsilon'|_{[\varepsilon^\alpha,L'_\varepsilon-\varepsilon^\alpha]}]+\mathcal{F}_{\varepsilon}[\gamma_\varepsilon'|_{[L'_\varepsilon-\varepsilon^\alpha,L'_\varepsilon]}],$$
	Lemma \ref{lemcurveconstruction} and Lemma \ref{lemcurveconstruction2} imply that the constructed curve $\gamma'_\varepsilon$ satisfies the energy convergence (\ref{eqnenergylimit}).
	In particular, we note that
	$$\mathcal{F}[\gamma'_\varepsilon|_{[L'_\varepsilon-\varepsilon^\alpha,L'_\varepsilon]}]=\mathcal{F}_\varepsilon[\gamma_\varepsilon^1|_{[0,\varepsilon^\alpha]}]$$
	since the combination of the backward reparameterization and the half-rotation for a curve maintains the value of $\mathcal{F}_\varepsilon$ (the translation also maintains $\mathcal{F}_\varepsilon$ obviously).
	The proof is now complete.
\end{proof}

\begin{figure}[tb]
	\centering
	\includegraphics[width=60mm]{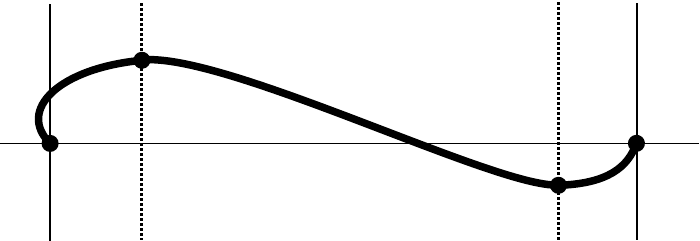}
	\caption{Construction of a curve for Proposition \ref{proplimsupinequality}.}
	\label{figupperbound2}
\end{figure}

\section{Convergence of minimizers}\label{sectconvergence}

In this section, we prove Theorem \ref{thmmain1} by using results in the previous section.
The rescaled convergence part is first proved in a weak sense, more precisely, the $H^2$-weak sense of curves.
The almost straightness part is then fully proved.
For these parts we mainly use properties of the energy.
After that, we improve the regularity of our rescaled convergence; in this regularity part we strongly use properties of the elastica equation.

\subsection{Rescaled weak convergence to borderline elasticae near the endpoints}

We first prove (1) of Theorem \ref{thmmain1} in a weak sense.
The following fact is an essential step of our proof.

\begin{lemma}\label{lemvartheta2}
	Let $c>0$ and $\vartheta\in H^1(0,c)$.
	Suppose that $\vartheta(0)\in[-\pi,\pi]$ and
	$$8\sqrt{2}\left(\sin^2\frac{\vartheta(0)}{4}-\sin^2\frac{\llbracket\vartheta(c)\rrbracket}{4}\right) \geq \int_0^{c} \left(|\vartheta'|^2+(1-\cos\vartheta)\right),$$
	where $\llbracket\cdot\rrbracket$ is defined in Remark \ref{remvariation}.
	Then, in the above inequality, the equality is attained.
	Moreover, if $|\vartheta(0)|<\pi$, the function $\vartheta$ is the borderline angle function with initial angle $\vartheta(0)$ (in the sense of Definition \ref{defborderline}).
	In the case that $|\vartheta(0)|=\pi$, up to the addition of a constant in $\{0,\pm2\pi\}$, the function $\vartheta$ is either the borderline angle function with initial angle $\pi$ or $-\pi$.
\end{lemma}

\begin{proof}
	By the inequality $X^2+Y^2\geq 2|X||Y|$,
	\begin{eqnarray}
	\int_0^{c} \left(|\vartheta'|^2+(1-\cos\vartheta)\right) \geq \int_0^{c} 2|\vartheta'|\sqrt{1-\cos\vartheta} = \int_0^{c} |(V\circ\vartheta)'|. \label{eqnlem1}
	\end{eqnarray}
	By the triangle inequality,
	\begin{eqnarray}\label{eqnlem2}
	\int_0^{c} |(V\circ\vartheta)'| \geq |V(\vartheta(0))-V(\vartheta(c))|.
	\end{eqnarray}
	Moreover, by Lemma \ref{lemvartheta},
	\begin{eqnarray}
	|V(\vartheta(0))-V(\vartheta(c))| &\geq& 8\sqrt{2}\left|\sin^2\frac{\llbracket\vartheta(0)\rrbracket}{4}-\sin^2\frac{\llbracket\vartheta(c)\rrbracket}{4}\right| \label{eqnlem3}\\
	&\geq& 8\sqrt{2}\left(\sin^2\frac{\vartheta(0)}{4}-\sin^2\frac{\llbracket\vartheta(c)\rrbracket}{4}\right). \label{eqnlem4}
	\end{eqnarray}
	The last inequality follows by the definition of absolute value and the assumption that $\vartheta(0)\in[-\pi,\pi]$, i.e., $|\llbracket\vartheta(0)\rrbracket|=|\vartheta(0)|$.

	Then, by the assumption, it turns out that in all the above inequalities (\ref{eqnlem1})--(\ref{eqnlem4}) the equalities are attained.
	The equality in (\ref{eqnlem1}) implies $|\vartheta'|^2=1-\cos\vartheta$ for a.e.\ in $[0,c]$.
	The equality in (\ref{eqnlem2}) implies that $(V\circ\vartheta)'$ does not change the sign, i.e., $\vartheta$ is monotone.
	Thus, $\vartheta$ satisfies either of the equations (\ref{eqnborderline}) in the classical sense.

	By the above fact, the proof is complete when $\vartheta(0)=0$ since the solution of (\ref{eqnborderline}) is unique in this case.
	Moreover, if $|\vartheta(0)|=\pi$, we also obtain the assertion by noting the symmetry of the solutions.
	In the case that $0<|\vartheta(0)|<\pi$, there are still two possibilities on $\vartheta$ since there are two solutions to (\ref{eqnborderline}).
	One solution is the desired borderline angle function; in this case the function $|\vartheta|$ is strictly decreasing.
	The other one corresponds to the case that $|\vartheta|$ is strictly increasing.
	However, since $\vartheta(0)\in(-\pi,\pi)$, Lemma \ref{lemvartheta} and the equality in (\ref{eqnlem3}) imply that $\vartheta(c)\in[-\pi,\pi]$.
	In addition, by the equality in (\ref{eqnlem4}) and the fact that $\llbracket\vartheta(0)\rrbracket=\vartheta(0)\in(-\pi,\pi)$, we find that $|\llbracket\vartheta(c)\rrbracket|\leq|\vartheta(0)|$.
	In particular, $|\llbracket\vartheta(c)\rrbracket|<\pi$, and hence $\llbracket\vartheta(c)\rrbracket=\vartheta(c)$.
	Consequently, $|\vartheta(c)|\leq|\vartheta(0)|$.
	Thus the function $|\vartheta|$ is decreasing, and hence $\vartheta$ is nothing but the borderline angle function with initial angle $\vartheta(0)$.
	The proof is now complete.
\end{proof}

We are now in a position to prove the (weak) rescaled convergence.
We prove it in terms of the tangential angle.

\begin{proposition}\label{proprescaledconvergence}
	Let $\{\gamma_\varepsilon\}_\varepsilon$ be a sequence as in Theorem \ref{thmmain1}.
	Let $\tilde{\gamma}_\varepsilon$ be the arc length parameterization of $\gamma_\varepsilon$.
	Let $\vartheta_{\tilde{\gamma}_\varepsilon}$ be the unique tangential angle such that $\vartheta_{\tilde{\gamma}_\varepsilon}(0)=\theta^\varepsilon_0$.
	Fix any $c>0$.
	Define the rescaled tangential angle $\hat{\vartheta}_\varepsilon\in C^\infty([0,c])$ as $\hat{\vartheta}_\varepsilon(\hat{s}):=\vartheta_{\tilde{\gamma}_\varepsilon}(\varepsilon\hat{s})$ for any small $\varepsilon$ (so that $\varepsilon c<l_\varepsilon$).
	Then, for any subsequence of $\{\hat{\vartheta}_\varepsilon\}_\varepsilon$ there is a subsequence $\{\hat{\vartheta}_{\varepsilon'}\}_{\varepsilon'}$ such that $\hat{\vartheta}_{\varepsilon'}$ converges to some $\vartheta_*\in H^1(0,c)$ weakly in $H^1(0,c)$.

	Moreover, if $|\theta_0|<\pi$, the function $\vartheta_*$ is the (unique) borderline angle function with initial angle $\theta_0$ (in the sense of Definition \ref{defborderline}), and hence the convergence is valid in the full convergence sense.
	If $|\theta_0|=\pi$, up to the addition of a constant in $\{0,\pm2\pi\}$, the function $\vartheta_*$ is either the borderline angle function with initial angle $\pi$ or $-\pi$.
\end{proposition}

\begin{proof}
	We decompose the curve $\tilde{\gamma}_\varepsilon(s)$ into the part $s\in[0,c\varepsilon]$ and $s\in[c\varepsilon,L_\varepsilon]$ (where $L_\varepsilon=\mathcal{L}[\gamma_\varepsilon]$).
	By Lemma \ref{lemenergydecomposition}, the energy $\mathcal{F}_\varepsilon[\gamma_\varepsilon]$ is also decomposed as
	\begin{eqnarray}\label{eqnlowerbound3}
	\mathcal{F}_\varepsilon[\gamma_\varepsilon] = \mathcal{F}_\varepsilon[\tilde{\gamma}_\varepsilon|_{[0,c\varepsilon]}]+\mathcal{F}_\varepsilon[\tilde{\gamma}_\varepsilon|_{[c\varepsilon,L_\varepsilon]}].
	\end{eqnarray}
	By Lemma \ref{lemasymptitoticmodified}, the energy convergence (\ref{eqnenergylimit}) holds.
	Moreover, since $\hat{\vartheta}_\varepsilon(0)=\theta^\varepsilon_0\to\theta_0$ and
	\begin{eqnarray}
	\mathcal{F}_\varepsilon[\gamma_\varepsilon] \geq \mathcal{F}_\varepsilon[\tilde{\gamma}_\varepsilon|_{[0,c\varepsilon]}] = \int_0^{c} \left(|\partial_{\hat{s}}\hat{\vartheta}_\varepsilon|^2+(1-\cos\hat{\vartheta}_\varepsilon)\right)d\hat{s} \geq \int_0^{c} |\partial_{\hat{s}}\hat{\vartheta}_\varepsilon|^2d\hat{s}, \nonumber
	\end{eqnarray}
	the sequence $\{\hat{\vartheta}_\varepsilon\}_\varepsilon$ is bounded in $H^1(0,c)$ as $\varepsilon\to0$.
	Thus, for any subsequence there is a subsequence (without relabeling) such that $\hat{\vartheta}_\varepsilon$ weakly converges to some function $\vartheta_*\in H^1(0,c)$ as $\varepsilon\to0$, and hence $\hat{\vartheta}_\varepsilon$ uniformly converges to $\vartheta_*$ in $[0,c]$ by the Sobolev embedding.

	We next prove
	\begin{eqnarray}\label{eqnlowerbound4}
	\liminf_{\varepsilon\to0}\mathcal{F}_\varepsilon[\tilde{\gamma}_\varepsilon|_{[c\varepsilon,L_\varepsilon]}]\geq 8\sqrt{2}\left(\sin^2\frac{\llbracket\vartheta_*(c)\rrbracket}{4}+\sin^2\frac{\theta_1}{4}\right).
	\end{eqnarray}
	Notice that $\vartheta_{\tilde{\gamma}_\varepsilon}(c\varepsilon)$ $(=\hat{\vartheta}_\varepsilon(c))$ converges to $\vartheta_*(c)$ as $\varepsilon\to0$ since $\hat{\vartheta}_\varepsilon$ uniformly converges to $\vartheta_*$ in $[0,c]$.
	Moreover, by Proposition \ref{propsegmentconvergence} and Lemma \ref{lemmeanvalue}, there exists a sequence of $s_\varepsilon\in[c\varepsilon,L_\varepsilon]$ such that $\llbracket\vartheta_{\tilde{\gamma}_\varepsilon}(s_\varepsilon)\rrbracket\to0$ as $\varepsilon\to0$.
	Then, by Lemma \ref{lemenergydecomposition}, Lemma \ref{lemlowerbound}, and Lemma \ref{lemvartheta}, we find that
	\begin{eqnarray}
	&& \mathcal{F}_\varepsilon[\tilde{\gamma}_\varepsilon|_{[c\varepsilon,L_\varepsilon]}] = \mathcal{F}_\varepsilon[\tilde{\gamma}_\varepsilon|_{[c\varepsilon,s_\varepsilon]}]+\mathcal{F}_\varepsilon[\tilde{\gamma}_\varepsilon|_{[s_\varepsilon,L_\varepsilon]}] \nonumber\\
	&\geq& 8\sqrt{2}\left|\sin^2\frac{\llbracket\vartheta_{\tilde{\gamma}_\varepsilon}(c\varepsilon)\rrbracket}{4}-\sin^2\frac{\llbracket\vartheta_{\tilde{\gamma}_\varepsilon}(s_\varepsilon)\rrbracket}{4}\right| + 8\sqrt{2}\left|\sin^2\frac{\llbracket\vartheta_{\tilde{\gamma}_\varepsilon}(s_\varepsilon)\rrbracket}{4}-\sin^2\frac{\llbracket\vartheta_{\tilde{\gamma}_\varepsilon}(L_\varepsilon)\rrbracket}{4}\right|. \nonumber
	\end{eqnarray}
	Since $|\llbracket\vartheta_{\tilde{\gamma}_\varepsilon}(L_\varepsilon)\rrbracket|\to\theta_1$, taking the limit $\varepsilon\to0$, we obtain (\ref{eqnlowerbound4}).

	Combining the energy limit (\ref{eqnenergylimit}) with (\ref{eqnlowerbound3}) and (\ref{eqnlowerbound4}), we have
	\begin{eqnarray}
	8\sqrt{2}\left(\sin^2\frac{\theta_0}{4}-\sin^2\frac{\llbracket\vartheta_*(c)\rrbracket}{4}\right) &\geq& \limsup_{\varepsilon\to0}\mathcal{F}_\varepsilon[\gamma_\varepsilon|_{[0,c\varepsilon]}] \nonumber\\
	&=& \limsup_{\varepsilon\to0}\int_0^{c} \left(|\partial_{\hat{s}}\hat{\vartheta}_\varepsilon|^2+(1-\cos\hat{\vartheta}_\varepsilon)\right)d\hat{s}. \nonumber
	\end{eqnarray}
	Moreover, since $\hat{\vartheta}_\varepsilon$ converges to $\vartheta_*$ weakly in $H^1(0,c)$, we also have
	$$\liminf_{\varepsilon\to0}\int_0^{c} \left(|\partial_{\hat{s}}\hat{\vartheta}_\varepsilon|^2+(1-\cos\hat{\vartheta}_\varepsilon)\right)d\hat{s} \geq \int_0^{c} \left(|\vartheta'_*|^2+(1-\cos\vartheta_*)\right).$$
	Therefore, the function $\vartheta_*$ satisfies the assumption of Lemma \ref{lemvartheta2}, which implies the conclusion.
	The proof is complete.
\end{proof}

Since the endpoint $\gamma_\varepsilon(0)=(0,0)$ is fixed, we find that any sequence of minimizers converges to the borderline elastica in a weak sense.

\subsection{Almost straightness except the endpoints}

In this subsection, we prove (2) of Theorem \ref{thmmain1} by using the above weak convergence.
We improve the regularity of the weak convergence from the next subsection.

Since $|(\cos\theta,\sin\theta)-(1,0)|\leq|\theta|$ for $\theta \in [-\pi,\pi]$, we find that
$$|\partial_s\tilde{\gamma}_\varepsilon(s)-(1,0)|\leq|\llbracket\vartheta_{\tilde{\gamma}_\varepsilon}(s)\rrbracket|.$$
Hence, it suffices to prove the following proposition.

\begin{proposition}\label{propstraight}
	Let $\{\gamma_\varepsilon\}_\varepsilon$ be a sequence in Theorem \ref{thmmain1}.
	Let $\tilde{\gamma}_\varepsilon$ be the arc length parameterization of $\gamma_\varepsilon$.
	Fix any $c>0$.
	Let $K_{c\varepsilon}=[c\varepsilon,L_\varepsilon-c\varepsilon]$ for any small $\varepsilon$ (so that $\varepsilon<l_\varepsilon/c$), where $L_\varepsilon=\mathcal{L}[\gamma_\varepsilon]$.
	Then
	$$\limsup_{\varepsilon\to0}\max_{s\in K_{c\varepsilon}}|\llbracket\vartheta_{\tilde{\gamma}_\varepsilon}(s)\rrbracket|\leq 4 e^{-\frac{c}{\sqrt{2}}}.$$
\end{proposition}

\begin{proof}
	By Proposition \ref{proprescaledconvergence} and symmetry, the angles $\llbracket\vartheta_{\tilde{\gamma}_\varepsilon}(c\varepsilon)\rrbracket$ and $\llbracket\vartheta_{\tilde{\gamma}_\varepsilon}(L_\varepsilon-c\varepsilon)\rrbracket$ converge as $\varepsilon\to0$, and moreover
	$$\lim_{\varepsilon\to0}|\llbracket\vartheta_{\tilde{\gamma}_\varepsilon}(c\varepsilon)\rrbracket|=|\llbracket\vartheta_B^{\theta_0}(c)\rrbracket|\leq|\theta_0|,\quad \lim_{\varepsilon\to0}|\llbracket\vartheta_{\tilde{\gamma}_\varepsilon}(L_\varepsilon-c\varepsilon)\rrbracket|=|\llbracket\vartheta_B^{\theta_1}(c)\rrbracket|\leq|\theta_1|,$$
	where $\vartheta_B^{\theta_i}$ is the borderline angle function with initial angle $\theta_i$ for $i=0,1$.
	Notice that
	$$|\llbracket\vartheta_B^{\theta_i}(c)\rrbracket|\leq |\llbracket\vartheta_B^{\pi}(c)\rrbracket|=|\bar{\varphi}_\pm(c)\mp2\pi|=|\bar{\varphi}_\pm(-c)|=4\arctan\left(e^{-\frac{c}{\sqrt{2}}}\right).$$
	by the representation (\ref{solborderline}).
	Since $\arctan X\leq X$ for $X\geq0$, we see that, for $i=0,1$, $$|\llbracket\vartheta_B^{\theta_i}(c)\rrbracket|\leq 4 e^{-\frac{c}{\sqrt{2}}}.$$

	Thus it suffices to prove that
	$$\limsup_{\varepsilon\to0}\max_{s\in K_{c\varepsilon}}|\llbracket\vartheta_{\tilde{\gamma}_\varepsilon}(s)\rrbracket|=\max\{|\llbracket\vartheta_B^{\theta_0}(c)\rrbracket|,|\llbracket\vartheta_B^{\theta_1}(c)\rrbracket|\}=:\theta^*_c.$$
	Note that $\theta^*_c\in(0,\pi)$.
	We prove it by contradiction, so we assume that there would exist $\delta\in(0,\pi-\theta^*_c)$, a sequence $\varepsilon_j\to0$, and $s_j\in \mathring{K}_{c\varepsilon_j}:=(c\varepsilon_j,L_{\varepsilon_j}-c\varepsilon_j)$ such that
	$$\lim_{j\to\infty}|\llbracket\vartheta_{\tilde{\gamma}_{\varepsilon_j}}(s_j)\rrbracket|=\theta^*_c+\delta\in(\theta^*_c,\pi).$$
	By taking a subsequence if necessary, we may assume that $s_j$ converges.
	Then, by Proposition \ref{propsegmentconvergence} and Lemma \ref{lemmeanvalue}, there is a sequence of $s_j^*\in \mathring{K}_{c\varepsilon_j}$ such that $s_j^*\neq s_j$ and $\llbracket\vartheta_{\tilde{\gamma}_{\varepsilon_j}}(s_j^*)\rrbracket\to0$ as $j\to\infty$.
	We then cut the arc length interval $[0,L_{\varepsilon_j}]$ at the points $c\varepsilon_j$, $s_j$, $s_j^*$ and $L_{\varepsilon_j}-c\varepsilon_j$ and decompose the curve $\gamma_{\varepsilon_j}$ into the corresponding five parts.
	(Note that the order of $s_j$ and $s_j^*$ may change as $j\to\infty$.)
	By using Lemma \ref{lemlowerbound} and Lemma \ref{lemvartheta} for each of the parts and applying Lemma \ref{lemenergydecomposition}, we find that
	$$\liminf_{j\to\infty}\mathcal{F}_{\varepsilon_j}[\gamma_{\varepsilon_j}]\geq 8\sqrt{2}\left( \sin^2\frac{\theta_0}{4}+\sin^2\frac{\theta_1}{4}-2\sin^2\frac{\theta^*_c}{4}+2\sin^2\frac{\theta^*_c+\delta}{4} \right).$$
	However, this contradicts the energy convergence (\ref{eqnenergylimit}).
	The proof is complete.
\end{proof}

\subsection{Jacobi elliptic functions and elastica equation}

In the rest of this section we improve the regularity of the weak convergence in Proposition \ref{proprescaledconvergence}.
To this end we use some properties of elliptic functions.
In this subsection we briefly recall some properties of elliptic functions, and expressions of solutions to the elastica equation in terms of elliptic functions.

We first recall that any minimizer satisfies the following elastica equation.

\begin{proposition}[e.g.\ \cite{Br92,Si08}]\label{propelasticaequation}
	Let $\gamma_\varepsilon$ be any minimizer of $\mathcal{E}_\varepsilon$ in $\mathcal{A}$ (with any boundary condition) and $\tilde{\gamma}$ be the arc length parameterization.
	Then its signed curvature $\kappa=\partial_s\vartheta_{\tilde{\gamma}}$ satisfies
	\begin{eqnarray}\label{eqnelastica}
	\varepsilon^2(2\partial^2_s\kappa+\kappa^3)-\kappa=0.
	\end{eqnarray}
\end{proposition}

It is well-known that any solution of the above equation is solved in terms of the Jacobi elliptic functions.
We briefly recall the definitions and some properties of elliptic functions (see e.g.\ \cite{La89} for details).

Let $F(\xi;k)$ be the incomplete elliptic integral of the first kind of modulus $k\in(0,1)$:
$$F(\xi;k):=\int_0^\xi\frac{dt}{\sqrt{1-t^2}\sqrt{1-k^2t^2}}.$$
Let $K(k)$ be the complete elliptic integral of the first kind, i.e., $K(k):=F(1;k)$.

The function $\sn(x,k)$ is defined so that $x=F(\sn(x,k);k)$ for $|x|\leq K(k)$, and $\sn(x,k)=-\sn(x+2K(k),k)$ for $x\in\mathbb{R}$.
Note that $\sn(\cdot,k)$ is an odd $2K(k)$-antiperiodic function and, in $[-K(k),K(k)]$, strictly increasing from $-1$ to $1$.

The function $\cn(x,k)$ is defined as a unique smooth function such that $\cn(0,k)=1$ and $\cn^2(x,k)+\sn^2(x,k)=1$ for $x\in\mathbb{R}$.
Note that $\cn(\cdot,k)$ is an even $2K(k)$-antiperiodic function and, in $[0,2K(k)]$, strictly decreasing from $1$ to $-1$.

The function $\dn(x,k)$ is defined as a unique smooth function such that $\dn(0,k)=1$ and $\dn^2(x,k)+k^2\sn^2(x,k)=1$.
Note that $\dn(\cdot,k)$ is a positive even $2K(k)$-periodic function and, in $[0,K(k)]$, strictly decreasing from $1$ to $\sqrt{1-k^2}$.

For $k=0$, the functions $\sn$, $\cn$, $\dn$ are interpreted as $\sin$, $\cos$, $1$, respectively.
For $k=1$, they are interpreted as $\tanh$, $\sech$, $\sech$, respectively.

The following derivative formulae hold: for $k\in[0,1]$,
\begin{eqnarray}\label{eqnellipticderivative}
\sn'=\cn\dn,\quad \cn'=-\sn\dn,\quad \dn'=-k^2\sn\cn.
\end{eqnarray}

We finally recall that any solution to the equation (\ref{eqnelastica}) is expressed by an elliptic function.

\begin{proposition}[{e.g.\ \cite{Li96}}]\label{propelasticasolution}
	For any given $\varepsilon>0$ and initial values $\kappa(0)=a_0$ and $\partial_s\kappa(0)=b_0$, the equation (\ref{eqnelastica}) is uniquely solved in $\mathbb{R}$.
	Moreover, the solution is given by either
	\begin{enumerate}
		\item[(1)] $\kappa(s)=A\cn(\alpha s+\beta,k)$, where $k\in[0,1]$ is modulus, $A\cn(\beta,k)=a_0$, $-A\alpha\sn(\beta,k)\dn(\beta,k)=b_0$, $A^2=4k^2\alpha^2$, and $\varepsilon^2(A^2-2\alpha^2)=1$, or
		\item[(2)] $\kappa(s)=A\dn(\alpha s+\beta,k)$, where $k\in[0,1]$ is modulus, $A\dn(\beta,k)=a_0$, $-A\alpha p^2\sn(\beta,k)\cn(\beta,k)=b_0$, $A^2=4\alpha^2$, and $\varepsilon^2(A^2-2\alpha^2k^2)=1$.
	\end{enumerate}
	If $(a_0^2-2\varepsilon^{-2})a_0^2+4b_0^2\geq0$ then the solution is (1), and otherwise (2).
\end{proposition}

Since $\|\cn\|_\infty=\|\dn\|_\infty=1$, the above solution $\kappa$ satisfies $\|\kappa\|_\infty\leq|A|$.
We call the number $|A|$ {\it virtual maximum} of $\kappa$, since the maximum $|A|$ may not be attained in a finite interval.
In this paper we do not use the relations on the initial values $a_0$ and $b_0$.
We also mention a small remark that, since $\varepsilon^2$ is now positive, in the case of $\cn$ the modulus has a lower bound as $k\in(1/\sqrt{2},1]$.

\subsection{Boundedness of higher derivatives}

For improving the regularity of the weak convergence in Proposition \ref{proprescaledconvergence}, it suffices to prove that any higher order derivative of the rescaled tangential angle is (locally) bounded as $\varepsilon\to0$.
We prove the boundedness by using the expression in terms of elliptic functions.

\begin{proposition}\label{proprescaledconvergence2}
	Let $c>0$ and $\hat{\vartheta}_\varepsilon\in C^\infty([0,c])$ be the rescaled tangential angle function in Proposition \ref{proprescaledconvergence} for $\varepsilon>0$ with $c\varepsilon<l_\varepsilon$.
	Then for any positive integer $k$ the sequence of $\hat{\vartheta}_\varepsilon$ is bounded in $C^k([0,c])$ as $\varepsilon\to0$.
	Accordingly, the $H^1$-weak convergence in Proposition \ref{proprescaledconvergence} is improved to the $C^\infty$-convergence.
\end{proposition}

\begin{proof}
	Let $\kappa_\varepsilon(s)=\partial_s\vartheta_{\tilde{\gamma}_\varepsilon}(s)$ be the signed curvature of the original minimizer $\gamma_\varepsilon$.
	Recall that $\kappa_\varepsilon$ satisfies (\ref{eqnelastica}).
	Hence, the rescaled curvature $\hat{\kappa}_\varepsilon$ defined by
	$$\hat{\kappa}_\varepsilon(\hat{s}):=\partial_{\hat{s}}\hat{\vartheta}_\varepsilon(\hat{s})=\varepsilon\kappa_\varepsilon(\varepsilon\hat{s})$$
	satisfies the normalized elastica equation:
	\begin{eqnarray}\nonumber
	2\partial^2_{\hat{s}}\hat{\kappa}_\varepsilon + \hat{\kappa}_\varepsilon^3 - \hat{\kappa}_\varepsilon=0.
	\end{eqnarray}
	By Proposition \ref{propelasticasolution}, the rescaled curvature $\hat{\kappa}_\varepsilon$ is of the form either (1) or (2) with $\varepsilon=1$.
	Thus, it suffices to prove that the virtual maximum $|\hat{A}_\varepsilon|$ of $\hat{\kappa}_\varepsilon$ and the coefficient $\hat{\alpha}_\varepsilon$ of the variable is bounded as $\varepsilon\to0$; in fact, by the derivative formulae (\ref{eqnellipticderivative}) and the fact that all the elliptic functions and modulus $\hat{k}_\varepsilon$ are bounded above by $1$, any derivative of $\hat{\kappa}_\varepsilon$ is bounded by a polynomial of $|\hat{A}_\varepsilon|$ and $|\hat{\alpha}_\varepsilon|$.
	Moreover, by the relations in Proposition \ref{propelasticasolution} (with $\varepsilon=1$), the boundedness of $|\hat{A}_\varepsilon|$ and of $|\hat{\alpha}_\varepsilon|$ are equivalent.
	Hence, it suffices to prove that $|\hat{A}_\varepsilon|$ is bounded as $\varepsilon\to0$.

	We now prove the boundedness by contradiction; suppose that a subsequence (not relabeled) of the virtual maximum $|\hat{A}_\varepsilon|$ of $\hat{\kappa}_\varepsilon$ diverges to infinity as $\varepsilon\to0$.
	We prove that this assumption contradicts the fact that the sequence of $\hat{\kappa}_\varepsilon$ is bounded in $L^2(0,c)$ (by Proposition \ref{proprescaledconvergence}).
	By the relations of constants in Proposition \ref{propelasticasolution} for $\hat{\kappa}_\varepsilon$, the assumption that $|\hat{A}_\varepsilon|\to\infty$ implies that only the case (1) occurs for any small $\varepsilon$.
	Hence, the following relations hold: $$\hat{\kappa}_\varepsilon(\hat{s})=\hat{A}_\varepsilon\cn(\hat{\alpha}_\varepsilon\hat{s}+\hat{\beta}_\varepsilon,\hat{k}_\varepsilon),\quad \hat{k}_\varepsilon^2=\frac{\hat{A}_\varepsilon^2}{2(\hat{A}_\varepsilon^2-1)},\quad \hat{\alpha}_\varepsilon^2=\frac{\hat{A}_\varepsilon^2-1}{2}.$$
	Then we calculate
	\begin{eqnarray}\nonumber
	\|\hat{\kappa}_\varepsilon\|_{L^2(0,c)}^2=
	\frac{\hat{A}_\varepsilon^2}{|\hat{\alpha}_\varepsilon|}\int_{\hat{\beta}_\varepsilon}^{\hat{\alpha}_\varepsilon c+\hat{\beta}_\varepsilon}|\cn(x,\hat{k}_\varepsilon)|^2dx.
	\end{eqnarray}
	Since $\hat{\alpha}_\varepsilon\to\infty$ and $\hat{k}_\varepsilon\to1/\sqrt{2}$, for any small $\varepsilon$ the interval $[\hat{\beta}_\varepsilon,\hat{\alpha}_\varepsilon c+\hat{\beta}_\varepsilon]$ includes one period $4K(\hat{k}_\varepsilon)$ of $\cn(x,\hat{k}_\varepsilon)$:
	$$\int_{\hat{\beta}_\varepsilon}^{\hat{\alpha}_\varepsilon c+\hat{\beta}_\varepsilon}|\cn(x,\hat{k}_\varepsilon)|^2dx\geq \int_{0}^{4K(\hat{k}_\varepsilon)}|\cn(x,\hat{k}_\varepsilon)|^2dx.$$
	By the dominated convergence theorem and $K(\hat{k}_\varepsilon)\to K(1/\sqrt{2})$, the right-hand term converges to a positive value, namely,
	$$\int_0^{4K(1/\sqrt{2})}|\cn(x,1/\sqrt{2})|^2dx.$$
	Since $\hat{A}_\varepsilon^2/|\hat{\alpha}_\varepsilon|\to\infty$, the $L^2$-norm $\|\hat{\kappa}_\varepsilon\|_{L^2(0,c)}$ diverges to infinity.
	This is a contradiction, and hence the boundedness part is proved.

	The improvement of the regularity of convergence is obvious since, by the boundedness of higher order derivatives, the Arzel\`{a}-Ascoli theorem implies the desired $C^\infty$-convergence.
	The proof is now complete.
\end{proof}

We shall complete the proof of Theorem \ref{thmmain1}.

\begin{proof}[Proof of Theorem \ref{thmmain1}]
	Let $\{\gamma_\varepsilon\}_\varepsilon$ be any sequence of minimizers as in the assumption.
	For the part (1), since the position of $\gamma_{\varepsilon}(0)$ is fixed at the origin, it suffices to prove (1) in terms of the tangential angles.
	This follows by Proposition \ref{proprescaledconvergence} and Proposition \ref{proprescaledconvergence2}.
	The almost straightness part (2) is proved in Proposition \ref{propstraight}, which is also in terms of the tangential angles.
	The proof is now complete.
\end{proof}

\section{Qualitative properties}\label{sectqualitativeproperty}

In this section we prove Theorem \ref{thmqualitativeelastica} and Theorem \ref{thmuniqueness} by using Theorem \ref{thmmain1}.
In this part we also use the expressions of the curvatures in terms of elliptic functions as in Proposition \ref{propelasticasolution}.

\subsection{Self-intersection}

We first confirm that any minimizer has no self-intersection in the limit $\varepsilon\to0$.
This is an almost direct consequence of Theorem \ref{thmmain1}.

\begin{proposition}\label{propselfintersection}
	Let $\{\gamma_\varepsilon\}_\varepsilon$ be any sequence of minimizers as in Theorem \ref{thmmain1}.
	Then there is $\bar{\varepsilon}>0$ such that for any $\varepsilon\in(0,\bar{\varepsilon})$ the curve $\gamma_{\varepsilon}$ has no self-intersection.
\end{proposition}

\begin{proof}
	Fix sufficiently large $c>0$ so that $4e^{-\frac{c}{\sqrt{2}}}<1$ and the $x$-component of $\gamma_B^{\theta_i}(c)$ is positive for $i=0,1$, where $\gamma_{B}^{\theta_i}$ is the borderline elastica with initial angle $\theta_i$.
	Decompose the domain of the arc length parameterized curve $\tilde{\gamma}_\varepsilon$ into $[0,c\varepsilon]$, $[c\varepsilon,L_\varepsilon-c\varepsilon]$, and $[L_\varepsilon-c\varepsilon,L_\varepsilon]$.
	Then, for any small $\varepsilon$, the curve $\tilde{\gamma}_\varepsilon$ has no self-intersection in each of the parts by Theorem \ref{thmmain1}.
	Moreover, for any small $\varepsilon$, the parts $\tilde{\gamma}_\varepsilon|_{(0,c\varepsilon)}$, $\tilde{\gamma}_\varepsilon|_{(c\varepsilon,L_\varepsilon-c\varepsilon)}$, $\tilde{\gamma}_\varepsilon|_{(L_\varepsilon-c\varepsilon,L_\varepsilon)}$ are respectively included in the sets
	$$\{x< \tilde{x}_\varepsilon(c\varepsilon)\},\quad\{\tilde{x}_\varepsilon(c\varepsilon)< x < \tilde{x}_\varepsilon(L_\varepsilon-c\varepsilon)\},\quad\{\tilde{x}_\varepsilon(L_\varepsilon-c\varepsilon)< x\},$$
	where $\tilde{x}_\varepsilon$ denotes the $x$-component of $\tilde{\gamma}_\varepsilon$.
	This implies that there is no self-intersection in the whole of $\tilde{\gamma}_\varepsilon$ for small $\varepsilon$.
\end{proof}

\subsection{Inflection point}

We next discuss the number of the inflection points, i.e., the sign changes of the curvature.
Recall that the curvatures of all nontrivial (non-straight) solution curves are represented by non-zero elliptic functions, and hence their sign changes are well-defined if $|\theta_0|+|\theta_1|>0$ (and $\varepsilon$ is small).
In particular, all the zeroes of the curvature except the endpoints are nothing but the sign changes.

The key step is to prove that the number of the inflection points are bounded above by one for any small $\varepsilon$.
To this end we first prove the upper bound condition except for the special case that $\theta_0=\theta_1=0$ (Proposition \ref{propinfleciton0}).
Since this result is sufficient to deal with the generic angle condition, we then obtain a result to determine the number of the inflection points (Proposition \ref{propinfleciton}).
After that, we give another approach to obtain the upper bound (Proposition \ref{propinfleciton1}), which is valid for any ``small angle'' case, in particular, even for the ``zero angle'' case excepted in Proposition \ref{propinfleciton0}.

First, we shall obtain the upper bound except for the zero angle case.
The rough strategy is as follows; if $\theta_0\neq0$ and minimizers would have two inflection points, then the curves would contain half-periods of inflectional elasticae away from the origin; however, the tangential angles near the origin have the variation nearly $|\theta_0|$ as $\varepsilon\to0$ and hence, in view of periodicity, the tangential angles would also have a uniformly positive variation in the ``middle''; this contradicts the almost straightness.

\begin{proposition}\label{propinfleciton0}
	Let $\{\gamma_\varepsilon\}_\varepsilon$ be any sequence of minimizers as in Theorem \ref{thmmain1}.
	Suppose that $|\theta_0|+|\theta_1|>0$.
	Then there is $\bar{\varepsilon}>0$ such that for any $\varepsilon\in(0,\bar{\varepsilon})$ the curve $\gamma_\varepsilon$ has at most one inflection point.
\end{proposition}

\begin{proof}
	By symmetry, we may assume that $\theta_0>0$ without loss of generality.
	We prove by contradiction so suppose that there is a sequence $\varepsilon_j\to0$ such that $\gamma_{\varepsilon_j}$ has at least two inflection points.
	Recall that the signed curvature $\kappa_{\varepsilon}$ of $\tilde{\gamma}_{\varepsilon}$ is represented by an elliptic function as in Proposition \ref{propelasticasolution}.
	Since $\kappa_{\varepsilon_j}$ has a zero (and $\kappa_\varepsilon\not\equiv 0$ by $\theta_0\neq0$), it is of the form
	$$\kappa_{\varepsilon_j}(s)=A_j\cn(\alpha_j s+\beta_j,k_j),$$
	where $k_j\in(0,1)$, $A_j\neq0$, and $\alpha_j\neq0$.
	We take the smallest two zeroes $s_0^j,s_2^j\in(0,L_{\varepsilon_j})$ of $\kappa_{\varepsilon_j}$ with $s_0^j<s_2^j$.
	By the $2K$-antiperiodicity of $\cn$, we find that
	$$s_2^j=s_0^j+2K(k_j)/|\alpha_j|.$$
	We now extend the curvature function $\kappa_{\varepsilon_j}(s)$ as a $2K(k_j)$-antiperiodic function defined for any $s\in\mathbb{R}$ by using the elliptic function $\cn$; we use the same notation $\kappa_{\varepsilon_j}$ for the extended curvature.
	Let
	$$s_{\pm1}^j:=s_0^j\pm K(k_j)/|\alpha_j|.$$
	By the periodicity of $\cn$, the curvature $\kappa_{\varepsilon_j}$ takes its maximum or minimum at $s_{\pm1}^j$.

	For simplicity we first consider the case that $\theta_0\neq\pi$, that is, $\theta_0\in(0,\pi)$.
	Take arbitrary large $c>0$.
	Since $\theta_0\in(0,\pi)$, Theorem \ref{thmmain1} implies that the rescaled curvature $\hat{\kappa}_\varepsilon(\hat{s}):=\varepsilon\kappa_{\varepsilon}(\varepsilon\hat{s})$, defined for $\hat{s}\in[0,c]$, smoothly converges to $\partial_s\vartheta_B^{\theta_0}$, where $$\partial_s\vartheta_B^{\theta_0}(\hat{s})=-\sqrt{2}\sech\left(\frac{\hat{s}+s_{\theta_0}}{\sqrt{2}}\right),$$
	and $s_{\theta_0}\geq0$ is a unique constant.
	Thus, for any small $\varepsilon_j$, the curvature $\kappa_{\varepsilon_j}$ is negative and increasing in $[0,c\varepsilon_j]$.
	Hence, for any small $\varepsilon_j$, the interval $[0,c\varepsilon_j]$ is included in $[s_{-1}^j,s_0^j]$.
	In particular, $s_{0}^j>c\varepsilon_j$.
	Moreover, we have $s_{0}^j-s_{-1}^j\geq c\varepsilon_j$, and hence $s_{2}^j-s_{1}^j\geq c\varepsilon_j$.
	Since $s_{2}^j< L_{\varepsilon_j}$, we also find that $s_{1}^j<L_{\varepsilon_j}-c\varepsilon_j$.
	Combining with $s_{0}^j>c\varepsilon_j$, we see that $[s^j_0,s^j_1]\subset[c\varepsilon_j,L_{\varepsilon_j}-c\varepsilon_j]$.
	Noting the periodicity of $\cn$ (and taking $\vartheta_{\tilde{\gamma}_{\varepsilon_j}}$ so that $\vartheta_{\tilde{\gamma}_{\varepsilon_j}}(0)=\theta_0^{\varepsilon_j}$), we have
	\begin{eqnarray}
	2\left(\limsup_{j\to\infty}\max_{s\in [c\varepsilon_j,L_{\varepsilon_j}-c\varepsilon_j]}|\vartheta_{\tilde{\gamma}_{\varepsilon_j}}(s)|\right) &\geq& \limsup_{j\to\infty} (|\vartheta_{\tilde{\gamma}_{\varepsilon_j}}(s_{1}^j)|+|\vartheta_{\tilde{\gamma}_{\varepsilon_j}}(s_{0}^j)|) \nonumber\\
	&\geq&
	\limsup_{j\to\infty} |\vartheta_{\tilde{\gamma}_{\varepsilon_j}}(s_{1}^j)-\vartheta_{\tilde{\gamma}_{\varepsilon_j}}(s_{0}^j)| \nonumber\\
	&=& \limsup_{j\to\infty} |\vartheta_{\tilde{\gamma}_{\varepsilon_j}}(s_{0}^j)-\vartheta_{\tilde{\gamma}_{\varepsilon_j}}(s_{-1}^j)| \nonumber\\
	&\geq& \lim_{j\to\infty}|\vartheta_{\tilde{\gamma}_{\varepsilon_j}}(c\varepsilon_j)-\vartheta_{\tilde{\gamma}_{\varepsilon_j}}(0)| \nonumber\\
	&=& |\vartheta_B^{\theta_0}(c)-\vartheta_B^{\theta_0}(0)|=\theta_0-\vartheta_B^{\theta_0}(c). \nonumber
	\end{eqnarray}
	The last term tends to $\theta_0>0$ as $c\to\infty$.
	This contradicts (2) in Theorem \ref{thmmain1}.

	In the remaining case that $\theta_0=\pi$, the rescaled curvature $\hat{\kappa}_\varepsilon$ converges either $\partial_s\vartheta_B^{\pi}$ or $\partial_s\vartheta_B^{-\pi}$ up to a subsequence of any subsequence; in particular, there is a subsequence of $\{\hat{\kappa}_{\varepsilon_j}\}_j$ converging to either of them.
	Hence, we obtain a contradiction in the same way as above.
	The proof is now complete.
\end{proof}

By using the above upper bound, we determine the exact number of the inflection points providing the generic angle condition.

\begin{proposition}\label{propinfleciton}
	Let $\{\gamma_\varepsilon\}_\varepsilon$ be any sequence of minimizers as in Theorem \ref{thmmain1}.
	Suppose the generic angle condition (\ref{eqngenericanglecondition}).
	If $\theta_0\theta_1>0$ (resp.\ $\theta_0\theta_1<0$), then there is $\bar{\varepsilon}>0$ such that for any $\varepsilon\in(0,\bar{\varepsilon})$ the curve $\gamma_\varepsilon$ has exact one inflection point (resp.\ no inflection point).
\end{proposition}

\begin{proof}
	By symmetry, we may assume that $\theta_0\in(0,\pi)$ without loss of generality.
	Let $\kappa_{\varepsilon}$ denote the curvature of a minimizer $\gamma_\varepsilon$.
	In the case that $\theta_0\theta_1<0$, we easily find that $\kappa_\varepsilon(0)\kappa_\varepsilon(1)<0$ for any small $\varepsilon$ by (1) in Theorem \ref{thmmain1}.
	Hence, $\kappa_{\varepsilon}$ has at least one sign change for any small $\varepsilon$.
	By Proposition \ref{propinfleciton0}, $\kappa_{\varepsilon}$ has exactly one sign change.
	In the case that $\theta_0\theta_1>0$, we similarly find that $\kappa_\varepsilon(0)\kappa_\varepsilon(1)>0$ for any small $\varepsilon$.
	Hence, $\kappa_{\varepsilon}$ has either no sign change or at least two sign changes.
	By Proposition \ref{propinfleciton0}, $\kappa_{\varepsilon}$ has no sign change.
	The proof is now complete.
\end{proof}

\begin{remark}\label{reminflectioncritical2}
	In the above proof, the case that $\theta_1=0$ is not treated due to the complexity.
	As mentioned in Remark \ref{reminflectioncritical}, even if $\theta_1=0$, we can also determine the number of the inflection points providing additional conditions, for example, $\theta_0>0$ and $\theta^\varepsilon_1\geq0$ for any small $\varepsilon$.
	In this case the curvature has exactly one sign change for any small $\varepsilon$.

	We shall confirm the above fact.
	We notice that, by (1) in Theorem \ref{thmmain1} and symmetry, the straightness (2) in Theorem \ref{thmmain1} extends to the endpoint $(l_\varepsilon,0)$, i.e., for any $c>0$,
	$$\limsup_{\varepsilon\to0}\max_{s\in [c\varepsilon,L_\varepsilon]}|\partial_s\tilde{\gamma}_\varepsilon(s)-(1,0)|\leq 4 e^{-c/\sqrt{2}}.$$
	Let $c>0$ be sufficiently large so that for any small $\varepsilon$ the $x$-component of $\partial_s\tilde{\gamma}_\varepsilon$ is positive in $[c\varepsilon,L_\varepsilon]$.
	By (1) in Theorem \ref{thmmain1}, the assumption that $\theta_0>0$ implies that the $y$-components of $\tilde{\gamma}_\varepsilon(c\varepsilon)$ and $\partial_s\tilde{\gamma}_\varepsilon(c\varepsilon)$ are positive for any small $\varepsilon$.
	Then the curve $\tilde{\gamma}_\varepsilon|_{[c\varepsilon,L_\varepsilon]}$ is represented as the graph of a function $u_\varepsilon$ defined on an interval $[a_\varepsilon,b_\varepsilon]$ such that
	$$u_\varepsilon(a_\varepsilon)>0,\quad u_\varepsilon'(a_\varepsilon)>0,\quad u_\varepsilon(b_\varepsilon)=0,\quad u'_\varepsilon(b_\varepsilon)\geq0.$$
	By this boundary condition, the second derivative $u_\varepsilon''$ must have a zero in $(a_\varepsilon,b_\varepsilon)$; in fact, if $u_\varepsilon''>0$ (resp.\ $u_\varepsilon''<0$), then the first two conditions contradict the third (resp.\ fourth) condition.
	Since a zero of $u_\varepsilon''$ corresponds to a sign change of $\kappa_{\varepsilon}$, we find that $\kappa_\varepsilon$ has a sign change for any small $\varepsilon$.
	By Proposition \ref{propinfleciton0}, $\kappa_{\varepsilon}$ has exact one sign change.
	The proof is complete.

	Note that in this proof the graph representation is essential.
	In particular, for any nonzero vectors $v_0,v_1\in\mathbb{R}^2$, there is a non-graph (looping) smooth regular curve $\gamma:\bar{I}\to\mathbb{R}^2$ without inflection point such that $\gamma(0)=(0,0)$, $\gamma(1)=(1,0)$, $\dot{\gamma}(0)=v_0$ and $\dot{\gamma}(1)=v_1$.
\end{remark}

\begin{remark}
	As mentioned in Remark \ref{reminflectioncritical} the critical case $|\theta_0|=\pi$ or $|\theta_1|=\pi$ is also excluded.
	However, in the special case that $|\theta_0|=\pi$ and $\theta^\varepsilon_1\equiv\theta_1=0$ (or left and right reversed), thanks to the symmetry of the $x$-axis reflection, the same argument as in Remark \ref{reminflectioncritical2} implies that any minimizer has exact one inflection point for small $\varepsilon$.
\end{remark}

Finally, we obtain the upper bound in a different way in small angle cases.
This case is relatively easy; it follows from Theorem \ref{thmmain1} and the fact that, for any inflectional elastica with tension $\lambda$ (where $\lambda$ is as in (\ref{eqnborderline})), the sign of tension $\lambda$ characterizes whether the {\em turning angle} is larger or smaller than $\pi$ as in Figure \ref{figturningangle}.
Here the turning angle means the oscillation of the tangential angle of a periodically extended inflectional elastica (or twice the angle between an inflectional elastica and the axis though inflection points as indicated in Figure \ref{figturningangle}).
The above fact is proved in, e.g., \cite[Corllary 5]{Br92}, \cite[Eq.\ (11) and Eq.\ (36)]{DjHaMlVa08}.
In particular, in our case $\lambda$ is nothing but $1/\varepsilon^2>0$, and hence any half-period connecting two adjacent inflection points cannot be represented as the graph of any function as in the right of Figure \ref{figturningangle}.

\begin{figure}[tb]
	\centering
	\includegraphics[width=100mm]{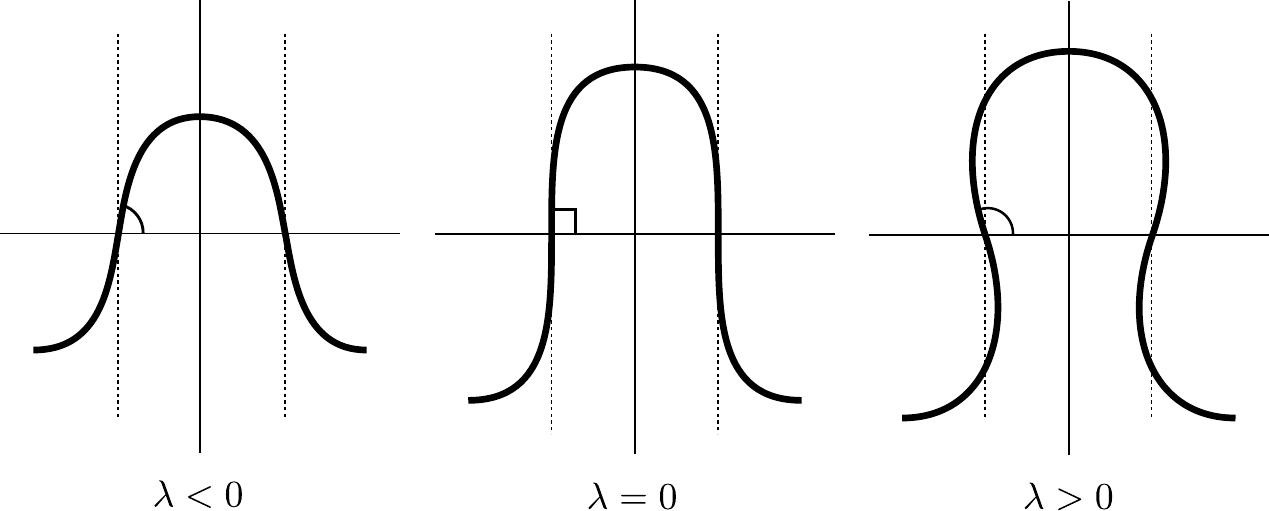}
	\caption{Relation between the turning angle of an inflectional elastica and the sign of tension $\lambda$.}
	\label{figturningangle}
\end{figure}

We shall complete the proof of the upper bound.
Recall that in our terminology the straight line has no inflection point.

\begin{proposition}\label{propinfleciton1}
	Let $\{\gamma_\varepsilon\}_\varepsilon$ be any sequence of minimizers as in Theorem \ref{thmmain1}.
	Suppose that $|\theta_0|,|\theta_1|<\pi/2$.
	Then there is $\bar{\varepsilon}>0$ such that for any $\varepsilon\in(0,\bar{\varepsilon})$ the curve $\gamma_\varepsilon$ has at most one inflection point.
\end{proposition}

\begin{proof}[Proof of Proposition \ref{propinfleciton1}]
	Since $|\theta_0|,|\theta_1|<\pi/2$, Theorem \ref{thmmain1} implies that there is $\bar{\varepsilon}>0$ such that for any $\varepsilon\in(0,\bar{\varepsilon})$ the minimizer $\gamma_\varepsilon$ is represented as the graph curve $\{y=u_\varepsilon(x)\}$ of some smooth function $u_\varepsilon$ up to isometry; this is easily proved by the same decomposition as in the proof of Proposition \ref{propselfintersection}.
	Then we find that such $\gamma_\varepsilon$ does not have two inflection points; in fact, if so, then the graph representation of $\gamma_\varepsilon$ contradicts the turning angle condition (as described above the statement of this proposition).
	The proof is complete.
\end{proof}

To complete the proof of Theorem \ref{thmqualitativeelastica} we shall summarize the results in this subsection.

\begin{proof}[Proof of Theorem \ref{thmqualitativeelastica}]
	Proposition \ref{propselfintersection} immediately implies the self-intersection part in Theorem \ref{thmqualitativeelastica} since if there would be a sequence $\varepsilon_j\to0$ and a sequence of minimizers $\{\gamma_{\varepsilon_j}\}_j$ having self-intersections, then it contradicts Proposition \ref{propselfintersection}.
	Similarly, Propositions \ref{propinfleciton0} and \ref{propinfleciton1} imply the upper bound part in Theorem \ref{thmqualitativeelastica}, and also, Proposition \ref{propinfleciton} implies the part to determine the number of the inflection points.
	Finally, combining Proposition \ref{propinfleciton} with Theorem \ref{thmmain1}, we immediately obtain the part on the total variation of the tangential angle in Theorem \ref{thmqualitativeelastica}.
	The proof is now complete.
\end{proof}

\subsection{Uniqueness}

We finally prove the uniqueness result as in Theorem \ref{thmuniqueness}.

For $l>0$ and $\theta_0,\theta_1\in\mathbb{R}$, we denote by $\tilde{\mathcal{A}}_{\theta_0,\theta_1,l}$ the set of all smooth constant speed curves joining $(0,0)$ to $(l,0)$ such that the tangential angles are strictly monotone functions from $\theta_0$ to $\theta_1$.
Notice that $\tilde{\mathcal{A}}_{\theta_0,\theta_1,l}\subset \mathcal{A}_{\theta_0,\theta_1,l}$ if $\theta_0,\theta_1\in[-\pi,\pi]$.
Notice that the constraint of $\tilde{\mathcal{A}}_{\theta_0,\theta_1,l}$ completely fixes the variation of the tangential angle of a curve unlike our original clamped boundary condition.
We also remark that $\tilde{\mathcal{A}}_{\theta_0,\theta_1,l}$ is possible to be empty for some angles $\theta_0,\theta_1$, but nonempty if e.g.\ $\theta_0\theta_1<0$.

The following statement is a key step for the proof.

\begin{proposition}\label{propuniqueness}
	Let $l>0$ and $\theta_0,\theta_1\in\mathbb{R}$.
	Then, for any $\varepsilon>0$ the energy $\mathcal{E}_\varepsilon:\tilde{\mathcal{A}}_{\theta_0,\theta_1,l}\to(0,\infty)$ admits at most one minimizer in $\tilde{\mathcal{A}}_{\theta_0,\theta_1,l}$.
\end{proposition}

To prove Proposition \ref{propuniqueness}, we convexify our minimizing problem by using the radius of curvatures parameterized by the (monotone) tangential angles.
As mentioned in the introduction, the main idea is classical and has already appeared in Born's stability analysis \cite{Born1906}.

\begin{proof}[Proof of Proposition \ref{propuniqueness}]
	Without loss of generality we may assume that $\theta_0<\theta_1$ and that $\tilde{\mathcal{A}}_{\theta_0,\theta_1,l}$ is nonempty.
	For any $\gamma\in\tilde{\mathcal{A}}_{\theta_0,\theta_1,l}$, we can define the radius of curvature function $\rho:[\theta_0,\theta_1]\to(0,\infty)$ parameterized by the tangential angle as $\rho(\phi):=1/\kappa(\vartheta_{\tilde{\gamma}}^{-1}(\phi))$, where $\tilde{\gamma}$ is the arc length parameterization of $\gamma$ and $\kappa(s)=\partial_s\vartheta_{\tilde{\gamma}}(s)$.
	For any $\varepsilon>0$ and $\gamma\in\tilde{\mathcal{A}}_{\theta_0,\theta_1,l}$, the energy $\mathcal{E}_\varepsilon$ is represented as
	$$\mathcal{E}_\varepsilon[\gamma]=\int_{0}^{\mathcal{L}[\gamma]}\left(\varepsilon^2\kappa^2+1\right)ds= \int_{\theta_0}^{\theta_1}\left(\frac{\varepsilon^2}{\rho}+\rho\right)d\phi=:\tilde{\mathcal{E}}_\varepsilon[\rho].$$
	In particular, for any fixed $\varepsilon$, the energy $\tilde{\mathcal{E}}_\varepsilon$ is strictly convex with respect to $\rho$ since $\rho>0$ and the integrand $f(\rho)=\varepsilon^2/\rho+\rho$ is strictly convex in $(0,\infty)$.
	Moreover, the constraints on the positions of $\gamma$ at the endpoints
	$$\int_{0}^{\mathcal{L}[\gamma]}\cos\vartheta_{\tilde{\gamma}}ds=l, \quad \int_{0}^{\mathcal{L}[\gamma]}\sin\vartheta_{\tilde{\gamma}}ds=0,$$
	are also expressed in terms of $\rho$ as
	\begin{eqnarray}\label{eqnboundaryconditionradius}
		\int_{\theta_0}^{\theta_1}\rho\cos\phi d\phi=l, \quad \int_{\theta_0}^{\theta_1}\rho\sin\phi d\phi=0.
	\end{eqnarray}
	Conversely, if a smooth function $\rho:[\theta_0,\theta_1]\to(0,\infty)$ is given as satisfying (\ref{eqnboundaryconditionradius}), then we can restore a unique curve in $\tilde{\mathcal{A}}_{\theta_0,\theta_1,l}$ of which radius of curvature parameterized by the tangential angle is equal to $\rho$.

	We now denote by $\tilde{\mathcal{R}}_{\theta_0,\theta_1,l}$ the set of all functions $\rho\in C^\infty([\theta_0,\theta_1];(0,\infty))$ satisfying (\ref{eqnboundaryconditionradius}).
	Clearly, the set $\tilde{\mathcal{R}}_{\theta_0,\theta_1,l}$ is convex.
	Moreover, by the above arguments, we find that the minimizing problem of $\mathcal{E}_\varepsilon: \tilde{\mathcal{A}}_{\theta_0,\theta_1,l}\to(0,\infty)$ is equivalent to the minimizing problem of $\tilde{\mathcal{E}}_\varepsilon:\tilde{\mathcal{R}}_{\theta_0,\theta_1,l}\to(0,\infty)$.
	More explicitly, there is a bijection $\Phi$ from $\tilde{\mathcal{R}}_{\theta_0,\theta_1,l}$ to $\tilde{\mathcal{A}}_{\theta_0,\theta_1,l}$ such that for any $\varepsilon>0$ and $\rho\in \tilde{\mathcal{R}}_{\theta_0,\theta_1,l}$ the equality $\mathcal{E}_\varepsilon[\Phi(\rho)]=\tilde{\mathcal{E}}_\varepsilon[\rho]$ holds.
	In addition, we easily find that the energy $\tilde{\mathcal{E}}_\varepsilon:\tilde{\mathcal{R}}_{\theta_0,\theta_1,l}\to(0,\infty)$ admits at most one minimizer since $\tilde{\mathcal{E}}_\varepsilon$ is a strictly convex functional defined on a convex set.
	Therefore, we also find that the energy $\mathcal{E}_\varepsilon:\tilde{\mathcal{A}}_{\theta_0,\theta_1,l}\to(0,\infty)$ admits at most one minimizer.
	The proof is now complete.
\end{proof}

We shall complete the proof of Theorem \ref{thmuniqueness}.

\begin{proof}[Proof of Theorem \ref{thmuniqueness}]
	By Theorem \ref{thmqualitativeelastica}, there is $\bar{\varepsilon}>0$ such that, for any $\varepsilon\in(0,\bar{\varepsilon})$ and any minimizer of $\mathcal{E}_\varepsilon$ in $\mathcal{A}_{\theta_0^\varepsilon,\theta_1^\varepsilon,l_\varepsilon}$, the tangential angle is strictly monotone from $\theta^\varepsilon_0$ to $\theta^\varepsilon_1$, that is, the curve $\gamma_\varepsilon$ belongs to $\tilde{\mathcal{A}}_{\theta_0^\varepsilon,\theta_1^\varepsilon,l_\varepsilon}$.
	Since $\tilde{\mathcal{A}}_{\theta_0^\varepsilon,\theta_1^\varepsilon,l_\varepsilon}$ is included in $\mathcal{A}_{\theta_0^\varepsilon,\theta_1^\varepsilon,l_\varepsilon}$, if a minimizer of $\mathcal{E}_\varepsilon$ in $\mathcal{A}_{\theta_0^\varepsilon,\theta_1^\varepsilon,l_\varepsilon}$ belongs to $\tilde{\mathcal{A}}_{\theta_0^\varepsilon,\theta_1^\varepsilon,l_\varepsilon}$, then it also minimizes $\mathcal{E}_\varepsilon$ in $\tilde{\mathcal{A}}_{\theta_0^\varepsilon,\theta_1^\varepsilon,l_\varepsilon}$.
	Therefore, Proposition \ref{propuniqueness} implies the desired uniqueness.
	The proof is complete.
\end{proof}

\begin{remark}
	As explained precisely in Appendix \ref{appendixexistence}, for any fixed $l>0$ and $\theta_0,\theta_1\in[-\pi,\pi]$, the set of admissible curves $\mathcal{A}_{\theta_0,\theta_1,l}$ is decomposed into the sets $\mathcal{A}_{\theta_0,\theta_1,l,m}$ by winding number $m\in\mathbb{Z}$.
	For each $m$, the set $\mathcal{A}_{\theta_0,\theta_1,l,m}$ is defined to fix the variation of the tangential angle as
	$$\vartheta_{\gamma}(1)-\vartheta_\gamma(0)=\theta_1-\theta_0+2\pi m.$$
	It is known that, for any inflectional elastica (i.e., $\cn$-solution) of finite length, the range of its tangential angle is included in an interval of which width is less than $2\pi$ (see e.g.\ \cite{Br92}).
	Hence, if $|m|>1$, then $|\vartheta_{\gamma}(1)-\vartheta_\gamma(0)|\geq 2\pi$, and hence any critical point in $\mathcal{A}_{\theta_0,\theta_1,l,m}$ must be a non-inflectional elastica (i.e., $\dn$-solution).
	Therefore, for $|m|>1$, by the same convexification as above, we find that $\mathcal{E}_\varepsilon$ admits a unique minimizer in $\mathcal{A}_{\theta_0,\theta_1,l,m}$.
	For $|m|\leq1$, there may be multiple candidates of minimizers.
\end{remark}

\section{Connection of inextensible and extensible problems}\label{sectrelation}

In this section we prove Theorem \ref{thmmain2} and Theorem \ref{thmmain2'}.
The relation between the problems (\ref{minprob1}) and (\ref{minprob2}) is not so trivial at the level of global minimizers.
As already mentioned, the case that $\theta_0=\theta_1=0$ is omitted since it is not possible to express the inextensible problem (\ref{minprob2}) in terms of the extensible problem (\ref{minprob1}).

\subsection{Length of minimizers of the modified total squared curvature}

We shall confirm some properties of the minimum values of energy and the lengths of minimizers in the extensible problem.
Throughout this subsection, we fix $l>0$ and $\theta_0,\theta_1\in[-\pi,\pi]$ with $|\theta_0|+|\theta_1|>0$, and denote $\mathcal{A}_{\theta_0,\theta_1,l}$ by $\mathcal{A}$ simply.

We first confirm basic properties of the minimum function $$m(\varepsilon)=\min_{\gamma\in\mathcal{A}}\mathcal{E}_\varepsilon[\gamma].$$
We extend the function $m$ to the origin as $m(0)=l$.

\begin{proposition}\label{propminimum}
	The minimum function $m$ is strictly increasing and continuous in $[0,\infty)$.
	Moreover, $m$ is locally semi-convex in $(0,\infty)$.
\end{proposition}

\begin{proof}
	First we note that $m(\varepsilon)>l$ for $\varepsilon>0$ and $m(\varepsilon)\to l$ as $\varepsilon\to0$ by Lemma \ref{lemasymptitoticmodified} and the assumption that $|\theta_0|+|\theta_1|>0$.
	Let $0<\varepsilon_0<\varepsilon_1$.
	By taking a minimizer $\gamma_1\in\mathcal{A}$ of $\mathcal{E}_{\varepsilon_1}$, we find the strict monotonicity $$m(\varepsilon_0)\leq\mathcal{E}_{\varepsilon_0}[\gamma_1]<\mathcal{E}_{\varepsilon_1}[\gamma_1]=m(\varepsilon_1).$$
	Moreover, for any $\varepsilon>0$ and $\delta\in\mathbb{R}$ with small $|\delta|$, taking any minimizer $\gamma_\varepsilon\in\mathcal{A}$ of $\mathcal{E}_\varepsilon$, we have
	$$m(\varepsilon+\delta) \leq \mathcal{E}_{\varepsilon+\delta}[\gamma_\varepsilon] = \mathcal{B}[\gamma_\varepsilon]\delta^2 + 2\varepsilon\mathcal{B}[\gamma_\varepsilon]\delta + m(\varepsilon).$$
	This relation and the monotonicity imply the remaining conclusions.
\end{proof}

We mention that the semi-convexity is not used at least in this paper.

Now we define a set-valued function $\tilde{L}$ as
\begin{eqnarray}\label{eqnminlength}
\tilde{L}(\varepsilon):=\{\mathcal{L}[\gamma] \mid \textrm{$\gamma\in \mathcal{A}$ is a minimizer of $\mathcal{E}_\varepsilon$} \}
\end{eqnarray}
for $\varepsilon\in(0,\infty)$, and extend $\tilde{L}$ to the origin by $\tilde{L}(0)=\{l\}$.
(Note that the definition depends on the constraints $l,\theta_0,\theta_1$.)
By the existence of minimizers (Appendix \ref{appendixexistence}), the set $\tilde{L}(\varepsilon)$ is nonempty for any $\varepsilon>0$.
Moreover, we notice that $\tilde{L}(\varepsilon)\subset(l,\infty)$ for $\varepsilon>0$.
In addition, we have the following

\begin{proposition}\label{proplength1}
	The set-valued function $\tilde{L}$ is nondecreasing in the sense that, for any $0\leq\varepsilon_0<\varepsilon_1$, any $L_0\in\tilde{L}(\varepsilon_0)$ and $L_1\in\tilde{L}(\varepsilon_1)$ satisfy $L_0\leq L_1$.
\end{proposition}

\begin{proof}
	Fix such $\varepsilon_0$, $\varepsilon_1$, $L_0$ and $L_1$.
	The case $\varepsilon_0=0$ is obvious since $m(\varepsilon_1)>l$ so we assume that $\varepsilon_0>0$.
	By the definition of $\tilde{L}$, for $i=0,1$, there is a minimizer $\gamma_i\in\mathcal{A}$ of $\mathcal{E}_{\varepsilon_i}$ with length $L_i$.
	Then, noting the minimality of $\gamma_0$ and $\gamma_1$, we have
	$$\mathcal{E}_{\varepsilon_0}[\gamma_0]\leq\mathcal{E}_{\varepsilon_0}[\gamma_1], \quad \mathcal{E}_{\varepsilon_1}[\gamma_1]\leq\mathcal{E}_{\varepsilon_1}[\gamma_0],$$
	that is,
	$$\varepsilon_0^2\mathcal{B}[\gamma_0]+L_0\leq\varepsilon_0^2\mathcal{B}[\gamma_1]+L_1, \quad \varepsilon_1^2\mathcal{B}[\gamma_1]+L_1\leq\varepsilon_1^2\mathcal{B}[\gamma_0]+L_0.$$
	Combining these inequalities, we obtain $(\varepsilon_1^2-\varepsilon_0^2)(L_1-L_0)\geq0$, which implies $L_0\leq L_1$.
\end{proof}

Recall that $\tilde{L}(\varepsilon)$ is nonempty for any $\varepsilon$.
Moreover, as in \cite{NoOk15}, it is known that $\tilde{L}(\varepsilon)$ is a finite set.
Hence, the following upper and lower envelopes of $\tilde{L}$, which are single-valued functions, are well-defined:
$$L^*(\varepsilon):=\max\{L \mid L\in\tilde{L}(\varepsilon)\},\quad L_*(\varepsilon):=\min\{L \mid L\in\tilde{L}(\varepsilon)\}.$$

\begin{proposition}\label{proplength2}
	The function $L^*$ (resp. $L_*$) is nondecreasing and upper (resp. lower) semicontinuous.
\end{proposition}

\begin{proof}
	Notice that the monotonicity in $[0,\infty)$ follows by Proposition \ref{proplength1}.
	Moreover, the continuity at the origin follows by the length convergence in Proposition \ref{propsegmentconvergence}.
	Hence, it suffices to prove the semicontinuity at any fixed $\varepsilon>0$.
	We prove only the upper semicontinuity since the lower semicontinuity follows by a similar argument.

	For any $\delta\in\mathbb{R}$ with small $|\delta|$, we take a minimizer $\gamma_{\varepsilon+\delta}\in\mathcal{A}$ of $\mathcal{E}_{\varepsilon+\delta}$ so that $L^*(\varepsilon+\delta)=\mathcal{L}[\gamma_{\varepsilon+\delta}]$.
	Then, since the sequence $\{\gamma_{\varepsilon+\delta}\}_\delta$ is $H^2$-bounded by their minimality, for any subsequence there is a subsequence $\{\gamma_{\varepsilon+\delta'}\}_{\delta'}$ converging to a regular $H^2$-curve $\gamma'$ weakly in $H^2$ and strongly in $C^1$; in particular, $\mathcal{L}[\gamma_{\varepsilon+\delta'}]\to\mathcal{L}[\gamma']$.
	Noting the $H^2$-weak lower semicontinuity of $\mathcal{E}_\varepsilon$ and Proposition \ref{propminimum}, we have
	$$\mathcal{E}_\varepsilon[\gamma']\leq\liminf_{\delta'\to0}\mathcal{E}_{\varepsilon}[\gamma_{\varepsilon+\delta'}]=\liminf_{\delta'\to0}\mathcal{E}_{\varepsilon+\delta'}[\gamma_{\varepsilon+\delta'}]=\liminf_{\delta'\to0}m(\varepsilon+\delta')=m(\varepsilon),$$
	which implies that $\gamma'$ is a minimizer of $\mathcal{E}_\varepsilon$ (in the $H^2$-framework, and hence $\gamma'$ is smooth by Appendix \ref{appendixexistence}).
	Then we find that $$\lim_{\delta'\to0}L^*(\varepsilon+\delta')=\lim_{\delta'\to0}\mathcal{L}[\gamma_{\varepsilon+\delta'}]=\mathcal{L}[\gamma']\leq L^*(\varepsilon),$$
	and hence we obtain the upper semicontinuity
	$$\limsup_{\delta\to0}L^*(\varepsilon+\delta)\leq L^*(\varepsilon)$$
	in the full limit sense.
	The proof is now complete.
\end{proof}

Combining Proposition \ref{proplength1} and Proposition \ref{proplength2}, we see that the set of jump points
$$J:=\{\varepsilon\in[0,\infty) \mid L^*(\varepsilon)>L_*(\varepsilon)\}=\{\varepsilon\in[0,\infty) \mid \tilde{L}(\varepsilon)\ \textrm{is not a singleton}\}$$
consists of at most countably many elements, and moreover for any open set $U\subset[0,\infty)\setminus J$ the function $L_*$ ($=L_*$) is a strictly increasing continuous function on $U$.

We finally confirm the first order expansion of the lengths of minimizers with respect to $\varepsilon$.

\begin{proposition}\label{propminlength}
	Any sequence of $L_\varepsilon\in\tilde{L}(\varepsilon)$ satisfies, as $\varepsilon\to0$,
	$$L_\varepsilon=l+4\sqrt{2}\left(\sin^2\frac{\theta_0}{4}+\sin^2\frac{\theta_1}{4}\right)\varepsilon+o(\varepsilon).$$
\end{proposition}

\begin{proof}
	Let $X_\varepsilon:=\sqrt{\varepsilon\mathcal{B}[\gamma_{\varepsilon}]}$ and $Y_\varepsilon:=\sqrt{(L_\varepsilon-l)/\varepsilon}$.
	By Lemma \ref{lemasymptitoticmodified},
	$$X_\varepsilon^2+Y_\varepsilon^2=\frac{\mathcal{E}_\varepsilon[\gamma_{\varepsilon}]-l}{\varepsilon}=8\sqrt{2}\left(\sin^2\frac{\theta_0}{4}+\sin^2\frac{\theta_1}{4}\right)+o(1)$$
	as $\varepsilon\to0$.
	Moreover, by the Cauchy-Schwarz inequality,
	\begin{eqnarray}
	2X_\varepsilon Y_\varepsilon &=& 2\left(\int_{0}^{L_\varepsilon}|\partial_s\vartheta_{\tilde{\gamma}_\varepsilon}|^2ds\right)^{1/2}\left(\int_{0}^{L_\varepsilon}(1-\cos\vartheta_{\tilde{\gamma}_\varepsilon})ds\right)^{1/2} \nonumber\\
	&\geq& \int_{0}^{L_\varepsilon}|\partial_s\vartheta_{\tilde{\gamma}_\varepsilon}|2\sqrt{1-\cos\vartheta_{\tilde{\gamma}_\varepsilon}}ds = \int_{0}^{L_\varepsilon}|\partial_s(V\circ\vartheta_{\tilde{\gamma}_\varepsilon})|ds. \nonumber
	\end{eqnarray}
	By Lemma \ref{lemmeanvalue}, there is a sequence of $s_\varepsilon\in[0,L_\varepsilon]$ such that $\llbracket\vartheta_{\tilde{\gamma}_\varepsilon}(s_\varepsilon)\rrbracket\to0$.
	Hence, by the triangle inequality and Lemma \ref{lemvartheta}, we find that
	\begin{eqnarray}
	2X_\varepsilon Y_\varepsilon &\geq& \int_{0}^{s_\varepsilon}|\partial_s(V\circ\vartheta_{\tilde{\gamma}_\varepsilon})|ds + \int_{s_\varepsilon}^{L_\varepsilon}|\partial_s(V\circ\vartheta_{\tilde{\gamma}_\varepsilon})|ds \nonumber\\
	&\geq& 8\sqrt{2}\left(\sin^2\frac{\theta_0}{4}+\sin^2\frac{\theta_1}{4}\right)-o(1) \nonumber
	\end{eqnarray}
	as $\varepsilon\to0$.
	Therefore, $0\leq(X_\varepsilon-Y_\varepsilon)^2\leq o(1)$ as $\varepsilon\to0$.
	Noting that $X_\varepsilon$ and $Y_\varepsilon$ are bounded as $\varepsilon\to0$, we find that $X_\varepsilon$ and $Y_\varepsilon$ converges to a same value up to a subsequence, and the fact that $X_\varepsilon^2+Y_\varepsilon^2$ converges implies the full convergence.
	Hence, we find that
	$$\frac{L_\varepsilon-l}{\varepsilon}=Y_\varepsilon^2\to4\sqrt{2}\left(\sin^2\frac{\theta_0}{4}+\sin^2\frac{\theta_1}{4}\right)$$
	as $\varepsilon\to0$.
	The proof is complete.
\end{proof}

\subsection{Connection of inextensible and extensible problems: fixed endpoints}

We prove a prototype of Theorem \ref{thmmain2}, which connects the inextensible problem to the extensible problem under a fixed clamped boundary condition.
This prototype deals with ``shortening'' ($L\downarrow l$) but not straightening ($l\uparrow L$); in the next subsection, we give a statement in terms of straightening.

\begin{proposition}\label{propeulermodified}
	Let $L>l$ and $\theta_0,\theta_1\in[-\pi,\pi]$ with $|\theta_0|+|\theta_1|>0$.
	Let $\tilde{L}$ be the length function (\ref{eqnminlength}) for $l,\theta_0,\theta_1$.
	Then, for any $\varepsilon>0$ such that $L\in\tilde{L}(\varepsilon)$, any minimizer of $\mathcal{B}$ in $\mathcal{A}_{\theta_0,\theta_1,l}^L$ is a minimizer of $\mathcal{E}_\varepsilon$ in $\mathcal{A}_{\theta_0,\theta_1,l}$.
\end{proposition}

\begin{proof}
	Let $\gamma$ be a minimizer of $\mathcal{B}$ in $\mathcal{A}_{\theta_0,\theta_1,l}^L$ and $\varepsilon>0$ satisfy $L\in\tilde{L}(\varepsilon)$.
	Then, by $L\in\tilde{L}(\varepsilon)$, there exists a minimizer $\gamma'$ of $\mathcal{E}_\varepsilon$ in $\mathcal{A}_{\theta_0,\theta_1,l}$ with $\mathcal{L}[\gamma']=L$ ($=\mathcal{L[\gamma]}$).
	Since $\gamma$ minimizes $\mathcal{B}$ in $\mathcal{A}_{\theta_0,\theta_1,l}^L$, we have $\mathcal{B}[\gamma]\leq\mathcal{B}[\gamma']$ and hence $\mathcal{E}_\varepsilon[\gamma]\leq \mathcal{E}_\varepsilon[\gamma']$.
	Since $\gamma'$ minimizes $\mathcal{E}_\varepsilon$ in $\mathcal{A}_{\theta_0,\theta_1,l}$, so does $\gamma$.
\end{proof}

We are now in a position to state the following Theorem \ref{thmmain3}, which ensures that the inextensible problem in the shortening limit is read as the extensible problem.

\begin{theorem}\label{thmmain3}
	Let $l>0$ and $\theta_0,\theta_1\in[-\pi,\pi]$ with $|\theta_0|+|\theta_1|>0$.
	Let $\tilde{L}$ be the length function (\ref{eqnminlength}) for $l,\theta_0,\theta_1$.
	Let $L_\varepsilon\downarrow l$ be a sequence such that there is $\bar{\varepsilon}>0$ such that $L_\varepsilon\in\tilde{L}(\varepsilon)$ for any $\varepsilon\in(0,\bar{\varepsilon})$.
	Then any minimizer $\gamma_\varepsilon$ of $\mathcal{B}$ in $\mathcal{A}_{\theta_0,\theta_1,l}^{L_\varepsilon}$ is a minimizer of $\mathcal{E}_{\varepsilon}$ in $\mathcal{A}_{\theta_0,\theta_1,l}$.
	Moreover, as $\varepsilon\to0$,
	$$\lim_{\varepsilon\to0}\frac{L_\varepsilon-l}{\varepsilon}=4\sqrt{2}\left(\sin^2\frac{\theta_0}{4}+\sin^2\frac{\theta_1}{4}\right).$$
\end{theorem}

\begin{proof}
	An immediate corollary of Proposition \ref{propminlength} and Proposition \ref{propeulermodified}.
\end{proof}

\subsection{Dilation}

We finally prove Theorem \ref{thmmain2} and Theorem \ref{thmmain2'} via Theorem \ref{thmmain3} and simple dilation arguments.
We use the following elementary facts, the proofs of which are omitted.

\begin{lemma}\label{lemdilation}
	Let $\theta_0,\theta_1\in[-\pi,\pi]$ and $0<\lambda<\Lambda$.
	Then a curve $\gamma$ is a minimizer of $\mathcal{B}$ in $\mathcal{A}_{\theta_0,\theta_1,\lambda}^{\Lambda}$ if and only if the curve $\frac{\Lambda}{\lambda}\gamma$ is a minimizer of $\mathcal{B}$ in $\mathcal{A}_{\theta_0,\theta_1,\Lambda}^{{\Lambda}^2/\lambda}$.
\end{lemma}

\begin{lemma}\label{lemdilation2}
	Let $\epsilon>0$, $\theta_0,\theta_1\in[-\pi,\pi]$ and $0<\lambda<\Lambda$.
	Then a curve $\gamma$ is a minimizer of $\mathcal{E}_\epsilon$ in $\mathcal{A}_{\theta_0,\theta_1,\Lambda}$ if and only if the curve $\frac{\lambda}{\Lambda}\gamma$ is a minimizer of $\mathcal{E}_{\lambda\epsilon/\Lambda}$ in $\mathcal{A}_{\theta_0,\theta_1,\lambda}$.
\end{lemma}

\begin{proof}[Proof of Theorem \ref{thmmain2}]
	Recall that the constants $L,\theta_0,\theta_1$ are given in the assumption.
	Let $\tilde{L}$ be the length function defined as (\ref{eqnminlength}) for $L,\theta_0,\theta_1$.
	Notice that $L'_\varepsilon\to L$ holds as $\varepsilon\downarrow 0$ for any sequence of $L'_\varepsilon\in\tilde{L}(\varepsilon)$ by Proposition \ref{propminlength}; in particular, there are sequences $L'_n\downarrow L$ and $\varepsilon_n\downarrow0$ as $n\to\infty$ such that $L'_n\in\tilde{L}(\varepsilon_n)$ for any $n$.
	Then, by Theorem \ref{thmmain3} with $l=L$, any minimizer of $\mathcal{B}$ in $\mathcal{A}_{\theta_0,\theta_1,L}^{L'_n}$ is a minimizer of $\mathcal{E}_{\varepsilon_n}$ in $\mathcal{A}_{\theta_0,\theta_1,L}$, and moreover
	$$\lim_{n\to\infty}\frac{L'_n-L}{\varepsilon_n}=4\sqrt{2}\left(\sin^2\frac{\theta_0}{4}+\sin^2\frac{\theta_1}{4}\right).$$

	We now define $l_n$ as $l_n:=L^2/L'_n$.
	We confirm that the sequences $l_n\uparrow L$ and $\varepsilon_n\downarrow0$ satisfy the desired properties.
	Let $\gamma_n$ be any minimizer of $\mathcal{B}$ in $\mathcal{A}_{\theta_0,\theta_1,l_n}^{L}$.
	By Lemma \ref{lemdilation} with $\lambda=l_n$ and $\Lambda=L$, the curve $\frac{L}{l_n}\gamma_n$ is a minimizer of $\mathcal{B}$ in $\mathcal{A}_{\theta_0,\theta_1,L}^{L'_n}$.
	Hence, by Theorem \ref{thmmain3}, the curve $\frac{L}{l_n}\gamma_n$ is a minimizer of $\mathcal{E}_{\varepsilon_n}$ in $\mathcal{A}_{\theta_0,\theta_1,L}$.
	Thus the first assertion is confirmed.
	Moreover, since $L'_n=L^2/l_n$ and $l_n/L\to1$, we have
	$$\lim_{n\to\infty}\frac{L-l_n}{\varepsilon_n}=\lim_{n\to\infty}\frac{L_n'-L}{\varepsilon_n}\cdot\frac{l_n}{L}=\lim_{n\to\infty}\frac{L_n'-L}{\varepsilon_n}=4\sqrt{2}\left(\sin^2\frac{\theta_0}{4}+\sin^2\frac{\theta_1}{4}\right),$$
	which is nothing but the last assertion.
	The proof is now complete.
\end{proof}

In the above proof we need to take a subsequence since the ``continuity'' of $\tilde{L}$ is not guaranteed in general even in a neighborhood of the origin.
Once the continuity is ensured, then there is no need to take a subsequence as shown in the following proof.

\begin{proof}[Proof of Theorem \ref{thmmain2'}]
	Recall that the constants $L,\theta_0,\theta_1$ with (\ref{eqngenericanglecondition}) and $\theta_0\theta_1<0$ are given in the assumption.
	By Theorem \ref{thmuniqueness}, there is $\bar{\varepsilon}>0$ such that for any $\varepsilon\in(0,\bar{\varepsilon})$ the energy $\mathcal{E}_\varepsilon$ admits a unique minimizer in $\mathcal{A}_{\theta_0,\theta_1,L}$.

	Let $\tilde{L}$ be the length function defined as (\ref{eqnminlength}) for fixed $L,\theta_0,\theta_1$.
	Then, by the above uniqueness, $\tilde{L}$ is a single-valued function in $[0,\bar{\varepsilon})$, and hence the lower semicontinuous envelope $L_*$ is a continuous nondecreasing function in $[0,\bar{\varepsilon}]$.
	Then, in particular, the function $L_*:[0,\bar{\varepsilon}]\to[L,L_*[\bar{\varepsilon}]]$ is surjective, and hence we can define a function $\tilde{\varepsilon}':[L,L_*(\bar{\varepsilon})]\to[0,\bar{\varepsilon}]$ so that $L_*\circ\tilde{\varepsilon}'$ is the identity map on $[L,L_*(\bar{\varepsilon})]$.
	Note that $\tilde{\varepsilon}'$ is a strictly increasing function since $L_*$ is nondecreasing.
	In addition, by Theorem \ref{thmmain3} with $l=L$, for any $L'\in(L,L_*(\bar{\varepsilon}))$, any minimizer of $\mathcal{B}$ in $\mathcal{A}_{\theta_0,\theta_1,L}^{L'}$ is a minimizer of $\mathcal{E}_{\tilde{\varepsilon}'(L')}$ in $\mathcal{A}_{\theta_0,\theta_1,L}$, and moreover
	$$\lim_{L'\downarrow L}\frac{L'-L}{\tilde{\varepsilon}'(L')}=4\sqrt{2}\left(\sin^2\frac{\theta_0}{4}+\sin^2\frac{\theta_1}{4}\right).$$
	In particular, for any $L'\in(L,L_*(\bar{\varepsilon}))$ the energy $\mathcal{B}$ admits a unique minimizer in $\mathcal{A}_{\theta_0,\theta_1,L}^{L'}$ (since $\mathcal{E}_{\tilde{\varepsilon}'(L')}$ admits a unique minimizer in $\mathcal{A}_{\theta_0,\theta_1,L}$).

	Now we set $\bar{l}:=L^2/L_*(\bar{\varepsilon})$.
	Define a function $\tilde{\varepsilon}:[\bar{l},L]\to[0,\bar{\varepsilon}]$ by $\tilde{\varepsilon}(l):=\tilde{\varepsilon}'(L^2/l)$.
	Notice that $\tilde{\varepsilon}$ is strictly decreasing.
	Then, by Lemma \ref{lemdilation}, for any $l\in(\bar{l},L)$ and any minimizer $\gamma_l$ of $\mathcal{B}$ in $\mathcal{A}_{\theta_0,\theta_1,l}^{L}$, the dilated curve $\frac{L}{l}\gamma_l$ minimizes $\mathcal{B}$ in $\mathcal{A}_{\theta_0,\theta_1,L}^{L^2/l}$.
	Since $L<L^2/l<L_*(\bar{\varepsilon})$, the desired uniqueness holds by the above arguments.
	In addition, we find that the curve $\frac{L}{l}\gamma_l$ also minimizes $\mathcal{E}_{\tilde{\varepsilon}(l)}$ in $\mathcal{A}_{\theta_0,\theta_1,L}$.
	Moreover, we also find that
	$$\lim_{l\uparrow L}\frac{L-l}{\tilde{\varepsilon}(l)}=\lim_{l\uparrow L}\frac{L^2/l-L}{\tilde{\varepsilon}'(L^2/l)}\cdot\frac{l}{L}=\lim_{L'\downarrow L}\frac{L'-L}{\tilde{\varepsilon}'(L')}=4\sqrt{2}\left(\sin^2\frac{\theta_0}{4}+\sin^2\frac{\theta_1}{4}\right).$$
	The proof is now complete.
\end{proof}

\begin{remark}
	It is not claimed that the above function $\tilde{\varepsilon}$ (or $\tilde{\varepsilon}'$) is continuous.
	The continuity is ensured if it is proved that the length function $\tilde{L}$ (or equivalently $L_*$) is strictly increasing.
\end{remark}

\appendix

\section{Existence of minimizers}
\label{appendixexistence}

Fix $l>0$ and $\theta_0,\theta_1\in[-\pi,\pi]$.
We say that $\gamma\in H^2(I;\mathbb{R}^2)\subset C^1(\bar{I};\mathbb{R}^2)$ is {\it $H^2$-admissible} if $\gamma$ is of constant speed and satisfying the boundary condition (\ref{eqnboundarycondition}).
We denote the set of $H^2$-admissible curves by $\mathcal{X}$.
Note that the $H^2$-weak topology is stronger than $C^1$-topology; hence, in particular, the set $\mathcal{X}$ is $H^2$-weakly closed in $H^2(I;\mathbb{R}^2)$.

In this $H^2$-framework we have an existence theorem of standard type:
Let $\mathcal{X}'\subset \mathcal{X}$ be an $H^2$-weakly closed subset.
Then the functional $\mathcal{E}_\varepsilon=\varepsilon^2\mathcal{B}+\mathcal{L}$ defined on $\mathcal{X}'$ attains its minimum in $\mathcal{X}'$.

The proof is straightforward.
Since any $\gamma\in \mathcal{X}'$ is of constant speed, we have the following representations:
$$\mathcal{L}[\gamma]\equiv|\dot{\gamma}|\geq l, \quad \mathcal{B}[\gamma]=\frac{1}{\mathcal{L}[\gamma]^3}\int_I|\ddot{\gamma}(t)|^2dt.$$
By the above relations and the boundary condition, we find that a minimizing sequence is $H^2$-bounded.
Since $\mathcal{E}_\varepsilon$ is lower semicontinuous with respect to the $H^2$-weak topology, a standard direct method implies the existence of a minimizer, completing the proof.

Moreover, if $\mathcal{X}'$ admits any local perturbation, then we find that any minimizer $\gamma\in\mathcal{X}'$ is of class $C^\infty$ by a bootstrap argument.
In particular, the problem (\ref{minprob1}) admits a smooth minimizer.
Using the Lagrange multiplier method to modify the length constraint, we find that the problem (\ref{minprob2}) also admits a smooth minimizer.
One may also refer to \cite[Theorem 2.2]{Li98} for a different argument.

In addition, it is also proved that there are infinitely many local minimizers with different winding numbers in a sense.
Here $\gamma\in\mathcal{X}$ is a local minimizer of the energy $\mathcal{E}_\varepsilon$ if there is $\delta>0$ such that $\mathcal{E}_\varepsilon[\gamma]\leq\mathcal{E}_\varepsilon[\gamma']$ for any $\gamma'\in\mathcal{X}$ with $\|\gamma-\gamma'\|_{H^2}\leq\delta$.
To state the above fact, we use a kind of winding number; for $\gamma\in\mathcal{X}$ we define $\mathcal{N}[\gamma]\in\mathbb{Z}$ as
$$\mathcal{N}[\gamma]=\frac{1}{2\pi}\left(\int_{\gamma}\kappa ds + \theta_0-\theta_1\right),$$
where $\kappa$ is the counterclockwise signed curvature ($\kappa=\partial_s\vartheta_{\tilde{\gamma}}$).
We notice that the functional $\mathcal{N}$ is $\mathbb{Z}$-valued and continuous with respect to the $H^2$-weak and -strong topologies.
Thus for any $m\in\mathbb{Z}$ the set $\mathcal{X}_m=\{\gamma\in \mathcal{X} \mid \mathcal{N}[\gamma]=m \}$ is open and closed in $\mathcal{X}$ both weakly and strongly.
Since $\mathcal{X}_m$ is weakly closed, the energies $\mathcal{E}_\varepsilon$ defined on $\mathcal{X}_m$ and $\mathcal{B}$ defined on $\mathcal{X}_m\cap \mathcal{X}^L$ attain their minimizers, where $L>l$ and $\mathcal{X}^L=\{\gamma\in\mathcal{X}\mid\mathcal{L}[\gamma]=L \}$.
Moreover, the set $\mathcal{X}_m$ is strongly open, and hence such minimizers are local minimizers on $\mathcal{X}$ or $\mathcal{X}^L$, respectively.

\section{Uniqueness of minimizers for well-prepared boundary conditions}
\label{appendixuniqueness}

In this section we briefly show that the uniqueness of global minimizers is easily proved or disproved for some special parameters of the boundary condition, which possess well-prepared symmetry.

We observe that under the generic boundary angle condition (\ref{eqngenericanglecondition}), if circular arcs are admissible, then global minimizers are unique.
For the extensible problem with any fixed $\varepsilon>0$, the inequality $X^2+Y^2\geq2XY$ leads to
$$\varepsilon^2\int_\gamma\kappa^2 ds + \int_\gamma ds \geq 2\varepsilon\int_\gamma|\kappa| ds,$$
where the equality is attained if and only if $|\kappa|=1/\varepsilon$.
In addition we easily observe that the right-hand side attains its minimum among admissible curves $\gamma\in\mathcal{A}_{\theta_0,\theta_1,l}$ if and only if a convex curve of rotation angle $|\theta_0|+|\theta_1|$ is admissible.
Therefore, if a circular arc of radius $1/\varepsilon$ is admissible, i.e., $\theta_0=-\theta_1$ and $l=2\cos\theta_0/\varepsilon$, then the circular arc of radius $1/\varepsilon$ is a unique global minimizer.
For the inextensible problem with fixed $L>0$, the Cauchy-Schwarz inequality leads to
$$\int_\gamma\kappa^2 ds \geq \frac{1}{L}\left(\int_\gamma|\kappa| ds\right)^2,$$
where the equality is attained if and only if $|\kappa|$ is constant.
Then a similar consideration implies that if a circular arc is admissible, i.e., $\theta_0=-\theta_1$ and $l=L\sin\theta_0/\theta_0$, then the circular arc of radius $L/2\theta_0$ is a unique global minimizer.

We shall finally think of some particular critical cases.
Both for the extensible problem with fixed $\varepsilon>0$ and the inextensible problem with fixed $L>0$, the most trivial case is that $(l,\theta_0,\theta_1)=(l,0,0)$ with $l>0$; in this case the segment is a unique minimizer.
If $(l,\theta_0,\theta_1)=(0,0,0)$ or $(l,|\theta_0|,|\theta_1|)=(0,\pi,\pi)$, a consideration as in the above paragraph implies there are only two minimizers of suitable circle, thus being not unique in the strict sense but unique up to symmetry.
In other critical cases, we are often able to ensure some nonuniqueness by a simple argument on symmetry, but up-to-symmetry uniqueness is a delicate issue (cf.\ \cite{SaSa14}).

\subsection*{Acknowledgments}

The author would like to thank Professor Yoshikazu~Giga, Professor Yasuhito~Miyamoto, Professor Michiaki~Onodera, and Dr.\ Olivier Pierre-Louis for their helpful comments and discussion.
In particular, Professor Miyamoto indicated to the author that our study is related to the works by Ni and Takagi.
The author learned of Audoly and Pomeau's book from Dr.\ Pierre-Louis, and found a description related to our study.
This work is mainly carried out at the University of Tokyo and partly at the Max Planck Institute for Mathematics in the Sciences.
This work was partially supported by JSPS KAKENHI Grant Numbers JP15J05166, JP18J30004, and also the Program for Leading Graduate Schools, MEXT, Japan.

\end{document}